\documentclass[11pt]{article}
\usepackage[pagebackref,colorlinks=true,urlcolor=blue,linkcolor=blue,citecolor=blue]{hyperref}

\usepackage{amsthm, amssymb, bm, bbm}
\usepackage{soul, color,cancel}
\usepackage[]{amsmath}
\usepackage[]{amsfonts}
\usepackage[]{fancyhdr}
\usepackage[]{graphicx}
\graphicspath{{/EPSF/}{../figures/}{./figures/}}

\newtheorem{theorem}{Theorem}[]
\newtheorem{lemma}[theorem]{Lemma}
\newtheorem{proposition}[theorem]{Proposition}

\newtheorem{remark}[theorem]{Remark}

\def \Cm {\mathbb{C}}
\def \Dmm {\mathbb{D}}
\def \Dm {M}

\def \Nm {\mathbb{N}}
\def \Rm {\mathbb{R}}
\def \Sm {\mathbb{S}}

\def \Zm {\mathbb{Z}}

\def\F{\mathcal{F}}
\def\H{\mathcal{H}}

\def\L{\mathcal{L}}

\def\SS{\mathcal{S}}

\def\V{\mathcal{V}}

\newcommand{\cout}[1]{}

\newcommand{\sgn}[1]{\,{\rm sign}(#1)}

\newcommand{\dprod}[2]{\langle{#1},{#2}\rangle}

\newcommand{\zbar}{\overline{z}}

\newcommand{\Aut}{\text{Aut} }

\def \ss {\mathfrak{s}}
\def \SS {\mathcal{S}}

\newcommand{\mat}[4]{\left[ \begin{array}{cc} #1 & #2 \\ #3 & #4 \end{array} \right]}

\renewcommand{\F}[1]{ #1} 

\newcommand{\smat}[4]{  \left[ \begin{smallmatrix} #1 & #2 \\ #3 & #4 \end{smallmatrix}  \right]}
\renewcommand{\mat}[4]{  \left[ \begin{array}{cc} #1 & #2 \\ #3 & #4 \end{array}  \right] }

\newcommand{\wtZ}{Z^\kappa}

\title{Range characterizations and Singular Value Decomposition of the geodesic X-ray transform on disks of constant curvature}
\author{Rohit Kumar Mishra\thanks{Department of Mathematics, University of Texas at Arlington, 655 W Mitchell Street, Arlington, TX 76010; email: rohit.mishra@uta.edu } \and Fran\c{c}ois Monard\thanks{Department of Mathematics, University of California Santa Cruz, 1156 High St, Santa Cruz, CA 95064; email: fmonard@ucsc.edu; The authors ackowledge funding from NSF grant DMS-1814104.}}

\begin{document}
\maketitle

\begin{abstract}
    For a one-parameter family of simple metrics of constant curvature ($4\kappa$ for $\kappa\in (-1,1)$) on the unit disk $M$, we first make explicit the Pestov-Uhlmann range characterization of the geodesic X-ray transform, by constructing a basis of functions making up its range and co-kernel. Such a range characterization also translates into moment conditions {\it \`a la} Helgason-Ludwig or Gel'fand-Graev. We then derive an explicit Singular Value Decomposition for the geodesic X-ray transform. Computations dictate a specific choice of weighted $L^2-L^2$ setting which is equivalent to the $L^2(M, dVol_\kappa)\to L^2(\partial_+ SM, d\Sigma^2)$ one for any $\kappa\in (-1,1)$. 
\end{abstract}

\section{Introduction}
Our object of study is the geodesic X-ray transform on a special family of simple surfaces. To give some context, fix a Riemannian surface $(M,g)$, with strictly convex boundary and no infinite-length geodesic. Denote its unit circle bundle $SM := \{ (x,v) \in TM,\ g_x(v,v) = 1\}$. The manifold of geodesics can then be modelled over the inward boundary $\partial_+ SM$ (points in $SM$ such that $x\in \partial M$ and $v$ points inwards), carrying the surface measure $d\Sigma^2$ inherited from the Sasaki volume form on $SM$. In this context, one defines the geodesic X-ray transform $I_0\colon C^\infty(M) \to C^\infty(\partial_+ SM)$ as 
\begin{align*}
    I_0 f(x,v) := \int_0^{\tau(x,v)} f(\gamma_{x,v}(t))\ dt, \qquad (x,v)\in \partial_+ SM,
\end{align*}
where $\gamma_{x,v}(t)$ is the unit-speed geodesic with $\gamma(0) = x$ and $\dot \gamma(0) = v$, and $\tau(x,v)$ is its first exit time. In integral geometry, one is concerned with the reconstruction of $f$ from knowledge of $I_0 f$, a problem with various generalizations (to tensor fields, general flows and sections of bundles), whose answer may depend on geometric features of the underlying metric, see \cite{Ilmavirta2019} for a recent topical review. Under the additional assumption that $M$ has no conjugate points\footnote{The three assumptions of convex boundary, no infinite-length geodesic, and no conjugate points, are summed up into the term {\em simple manifold}.}, positive answers to this problem can be provided, with varying degrees of explicitness. The problem is known to be injective in general \cite{Mukhometov1977}; the function $f$ can be reconstructed via explicit inversion formulas in constant curvature spaces \cite{Radon1917,Helgason1999}, and modulo compact error in variable curvature \cite{Pestov2004,Krishnan2010,Monard2013}. In \cite{Pestov2004}, a general range characterization of $I_0$ is given in terms of a 'boundary' operator $P_-$ (i.e., from a spaces of functions on $\partial_+ SM$ to itself), which was proved by the second author in \cite{Monard2015a} to be equivalent to the classical moment conditions (see Helgason-Ludwig \cite{Ludwig1966,Helgason1999} or Gel'fand-Graev \cite{GelfandGraev1960}) in the Euclidean case. 

Of crucial importance for practical purposes is the knowledge of the Singular Value Decomposition (SVD) of the operator $I_0$, be it for truncation and regularization purposes \cite{Natterer2001,Bertero1998}, to understand the structure of 'ghosts' in the case of discrete data \cite{Louis1984,Louis1985}, or to seek low-dimensional ansatzes in the case of incomplete data \cite{Louis1986,Louis1996}. Several results on the SVD of ray transforms have been obtained, mainly existing in the Euclidean case: on functions in \cite{Marr1974,Maass1987,Maass1991,Maass1992,Quinto1983,Natterer2001}, tensor fields in \cite{Kazantsev2004} and for the transverse ray transform in \cite{Derevtsov2011a}. Other transforms on circularly-symmetric families of curves have extensively been studied, see e.g. \cite{Cormack1981,Cormack1982,Rigaud2014}, though the literature on the SVD of the X-ray transform for families of geodesic curves remains scarce to the authors' knowledge. We present below a case where the SVD can be computed in a geodesic context with metrics of constant curvature $4\kappa$, $\kappa\in (-1,1)$, on the unit disk $\Dm = \{(x,y)\in \Rm^2,\ x^2+y^2\le 1\}$. \F{While an extension of the results to the case of ``Herglotz'' type metrics\footnote{By ``Herglotz'' type metric, we mean a scalar, rotation-invariant metric satisfying a non-trapping condition.} seems natural and of interest to the authors, the explicitness of the present results hinges on Lemmas \ref{lem:holo} and \ref{lem:sqrtjac} below, which at the moment take the form of calculations specific to constant curvature.} 

As the works \cite{Marr1974,Maass1991,Maass1992} show, even in the Euclidean case there are a few 'natural' choices of weighted $L^2-L^2$ settings to be decided upon, for which the SVD of $I_0$ may or may not be computationally tractable. The current generalization to Riemannian settings gives even more options of weights to be chosen for the target $L^2$ space, and somewhat surprisingly, the most 'tractable' codomain topology so far is $L^2(\partial_+ S\Dm, d\Sigma^2)$. In this case, the SVD functions obtained on $\Dm$ involve the Zernike polynomials \cite{Zernike1934}, up to some rational diffeomorphism and multiplication by an appropriate $\kappa$-dependent weight. The functions obtained are no longer polynomials, however. 

\F{Although the calculations of the present article are self-contained, several aspects of X-ray transforms motivate this work and the intuition behind it. A reader interested on aspects related to transport equations on the unit circle bundle, and/or microlocal aspects, may find relevant information in the expository paper \cite{Ilmavirta2019} and the references there. In some ways, the approach of the present paper follows that of \cite{Monard2015a}, where the X-ray transform on the Euclidean disk is treated. There, Euclidean geometry is nice enough that a full understanding of the X-ray transform defined on more general classes of integrands (vector fields and tensor fields) can be obtained, and the present results represent a first step towards achieving that same level of understanding on constant curvature spaces. 

Lastly, in their connection with inverse problems, an important motivation for our results is the following:} while it is documented that X-ray transforms are mildly ill-posed of order $1/2$ on simple surfaces, and severely ill-posed on some non-simple surfaces (see, e.g., \F{the works} \cite{Stefanov2012a,Monard2013b,Holman2017} \F{which address} the unconditional instability incurred by conjugate points), no analysis has been made of this transition of behavior as a metric evolves from simple to non-simple. The current article presents the first analysis that quantifies what happens as one approaches some borderline cases of simplicity, by fully describing the action of the geodesic X-ray transform along a one-parameter curve of metrics, whose endpoints are two such borderline cases (as $\kappa\to -1$, the manifold becomes non-compact; as $\kappa\to 1$, the manifold has conjugate points on its boundary, and the latter is also no longer convex). 

\paragraph{Main results.} As $I_0$ has infinite-dimensional co-kernel inside $L^2(\partial_+ SM, d\Sigma^2)$, we first endeavor to explicitly characterize this co-kernel. To this end, we use range characterization ideas coming from Pestov-Uhlmann \cite{Pestov2004} and refined in Proposition \ref{prop:PU} below. These range characterizations reframe the range of $I_0$ in terms of the range of an operator $P_-\in \L(L^2(\partial_+ SM, d\Sigma^2))$, or alternativaly in terms of the kernel of an operator $C_-\in \L(L^2(\partial_+ SM, d\Sigma^2))$ introduced in \cite{Monard2015a}. \F{These operators, initially motivated by how the fiberwise Hilbert transform acts on solutions of the geodesic transport equation inside $SM$ (see, e.g., \cite[\S 4]{Pestov2004}), admit a final expression solely in terms of ``boundary operators'', namely, the scattering relation and the fiberwise Hilbert transform on the fibers of $\partial SM$, given in  \eqref{eq:PC} below. As they are highly relevant in order to understand the range of $I_0$, yet their intuitive understanding is limited at this point, a workaround is to build their eigendecompositions in geometries where} the scattering relation can be explicitly worked out. Such an endeavor was first carried out in \cite{Monard2015a} in the case of the Euclidean disk, and a first salient feature of the present article is to generalize some of the results there, to the case of the unit disk equipped with the metric 
\begin{align}
    g_{\kappa}(z) = \left( 1 +\kappa |z|^2 \right)^{-2} |dz|^2, \qquad |z|\le 1,
    \label{eq:metric}
\end{align}
of constant curvature $4\kappa$ for any fixed $\kappa\in (-1,1)$. Specifically, we establish the singular value decomposition of the operators $P_-$ and $C_-$ when viewed as operators from $L^2(\partial_+ S\Dm, d\Sigma^2)$ into itself, see Theorem \ref{thm:SVD_PC} below. This in particular allows to formulate a few range characterizations of $I_0$. First note that as a function on $\partial_+ SM$, the X-ray transform of a function takes the same value whether one integrates from one end of a geodesic or the other. This gives a first symmetry, encapsulated by the map $\SS_A$ \eqref{eq:SA0}, mapping one end of a geodesic to the other. By $\SS_A^*$ we denote the pullback $\SS_A^* u := u\circ \SS_A$.  

\begin{theorem} \label{thm:range}
    Let $\Dm$ be equipped with the metric $g_\kappa$ \eqref{eq:metric} for $\kappa\in (-1,1)$ fixed. Suppose $u\in C^\infty(\partial_+ S\Dm)$ such that $\SS_A^* u = u$. Then the following conditions are equivalent: 
    
    (1) $u$ belongs to the range of $I_0\colon C^\infty(M)\to C^\infty(\partial_+ SM)$.

    (2) There exists $w\in C_{\alpha,+,-}^\infty(\partial_+ S\Dm)$ such that $u = P_- \F{w}$. 

    (3) $C_- u = 0$. 

    (4) $u$ satisfies a complete set of orthogonality/moment conditions: $(u, \psi^\kappa_{n,k})_{L^2(\partial_+ S\Dm, d\Sigma^2)} = 0$ for all $n\ge 0$ and $k$ such that $k<0$ or $k>n$, where in fan-beam coordinates,
    \begin{align*}
	\psi^\kappa_{n,k}(\beta,\alpha) &:= \frac{(-1)^n}{4\pi} \sqrt{\ss_\kappa'(\alpha)} e^{i(n-2k)(\beta+\ss_\kappa(\alpha))} (e^{i(n+1)\ss_\kappa(\alpha)} + (-1)^n e^{-i(n+1)\ss_\kappa(\alpha)}), \\
	\ss_\kappa(\alpha) &:= \tan^{-1} \left( \frac{1-\kappa}{1+\kappa}\tan \alpha \right).
    \end{align*}
\end{theorem}

\F{In the Euclidean case where $\kappa=0$, the functions $\psi_{n,k}^\kappa\equiv\psi_{n,k}$ are given in \eqref{eq:psinkeucl}, $\ss_\kappa (\alpha) = \alpha$, and the content of Theorem \ref{thm:range} is established in \cite[Theorem 2.3, \S 4]{Monard2015a}. Similarly to \cite[\S 4.4]{Monard2015a},} the characterization $(3)$ presents the advantage over $(2)$ that $C_-$ can be used to construct a projection operator (more precisely, $id+C_-^2$), allowing for example to project noisy data onto the range of $I_0$, see Theorem \ref{thm:projection} below. The orthogonality conditions $(4)$ are indexed over the eigenfunctions of $C_-$ associated with nontrivial eigenvalues. 

Now that Theorem \ref{thm:range} allows to isolate distinguished functions in $L^2(\partial_+ S\Dm, d\Sigma^2)$ which are orthogonal, and to accurately locate the range of $I_0$, one is then tempted to apply the adjoint for $I_0$ in this topology, and show that the functions so obtained are orthogonal for a specific choice of measure on $\Dm$, thereby finding the SVD of (some version of) $I_0$ in the process. The second salient feature of this article is to carry this agenda in full extent, adapting the Euclidean scenario (whose outcome produces the Zernike polynomials, presented as in \cite{Kazantsev2004}, see also Figure \ref{fig:Zernike} and Section \ref{sec:SVDE}), to the case of constant curvature disks. The method of proof consists in relating the case $\kappa \ne 0$ with the case $\kappa = 0$ by constructing diffeomorphisms on $M$ and $\partial_+ SM$ which intertwine the adjoints of $I_0$ associated with each geometry. To formulate the theorem, in addition to $\ss_\kappa(\alpha)$ and $\{\psi_{n,k}^\kappa\}_{n\ge 0, k\in \Zm}$, we also define 
\begin{align}
    \wtZ_{n,k} (z) := \sqrt{\frac{1-\kappa}{1+\kappa}} \frac{1+\kappa|z|^2}{1-\kappa|z|^2} Z_{n,k} \left( \frac{1-\kappa}{1-\kappa|z|^2} z \right), 
    \label{eq:Znkk}
\end{align}
where $Z_{n,k}$ are the Zernike polynomials in the convention of \cite{Kazantsev2004}. The radial profiles of the functions $Z_{n,k}^\kappa$ for low values of $n$ and $k$ are given Figure \ref{fig:radialPlots}. The family $\{\wtZ_{n,k}\}_{n\ge 0,\ 0\le k\le n}$ is a complete orthogonal system of $L^2(\Dm, w_\kappa\ dVol_\kappa)$ where $w_\kappa(z) := \frac{1+\kappa|z|^2}{1-\kappa|z|^2}$ with norm $\|Z_{n,k}\|^2 = \frac{1}{1-\kappa^2} \frac{\pi}{n+1}$. In addition, the family $\{\psi_{n,k}^\kappa\}_{n\ge 0,\ k\in \Zm}$ is a complete orthogonal system of the space $L^2(\partial_+ S\Dm, d\Sigma^2) \cap \ker(id-\SS_A^*)$, with norm $\|\psi_{n,k}^\kappa\|^2 = \frac{1}{4(1+\kappa)}$. We formulate our second main result as follows: 

\begin{theorem} Let $\Dm$ be the unit disk equipped with the metric $g_\kappa(z)$ defined in \eqref{eq:metric} for $\kappa\in (-1,1)$, with volume form $dVol_\kappa$. Let $\psi_{n,k}^\kappa$, $Z_{n,k}^\kappa$ defined as above and denote $\widehat{\wtZ_{n,k}}$ and $\widehat{\psi_{n,k}^\kappa}$ their normalizations in the respective spaces $L^2(M, w_\kappa\ dVol_\kappa)$ and $L^2(\partial_+ SM, d\Sigma^2)$. Then given any $f\in w_\kappa L^2(\Dm, w_\kappa\ dVol_\kappa)$, admitting a unique expansion
    \begin{align*}
	f = w_\kappa \sum_{n\ge 0} \sum_{k=0}^n f_{n,k} \widehat{\wtZ_{n,k}}, \qquad f_{n,k} := \left( f, \widehat{\wtZ_{n,k}} \right)_{L^2(\Dm, dVol_\kappa)}, \qquad \sum_{n,k} |f_{n,k}|^2 <\infty,
    \end{align*}
    we have
    \begin{align*}
	I_0 f = \sum_{n\ge 0} \sum_{k=0}^n \sigma^\kappa_{n,k}\ f_{n,k}\ \widehat{\psi^\kappa_{n,k}}, \qquad \sigma_{n,k}^\kappa := \frac{1}{\sqrt{1-\kappa}} \frac{2\sqrt{\pi}}{\sqrt{n+1}}. 
    \end{align*}
    In particular, the Singular Value Decomposition of $I_0 w_\kappa\colon L^2(M, w_\kappa\ dVol_\kappa)\to L^2(\partial_+ SM, d\Sigma^2)$ is $(\widehat{\wtZ_{n,k}}, \widehat{\psi_{n,k}^\kappa}, \sigma_{n,k}^\kappa)_{n\ge 0,\ 0\le k\le n}$.
    \label{thm:main}
\end{theorem}
The case $\kappa = 0$ recovers the Euclidean case\F{, where $Z_{n,k}^\kappa = Z_{n,k}$ (the Zernike polynomials as presented in \cite{Kazantsev2004}), $\psi_{n,k}^\kappa = \psi_{n,k}$ is given in \eqref{eq:psinkeucl} and $w_\kappa \equiv 1$}. The appearance of the weight $w_\kappa$ is a result of the method. For any $\kappa\in (-1,1)$, since $w_\kappa$ is bounded above and below by positive constants, the topologies $w_\kappa L^2(\Dm, w_\kappa\ dVol_\kappa)$ and $L^2(\Dm,\ dVol_\kappa)$ are equivalent.

\begin{figure}[htpb]
    \centering
    \includegraphics[height=0.3\textheight]{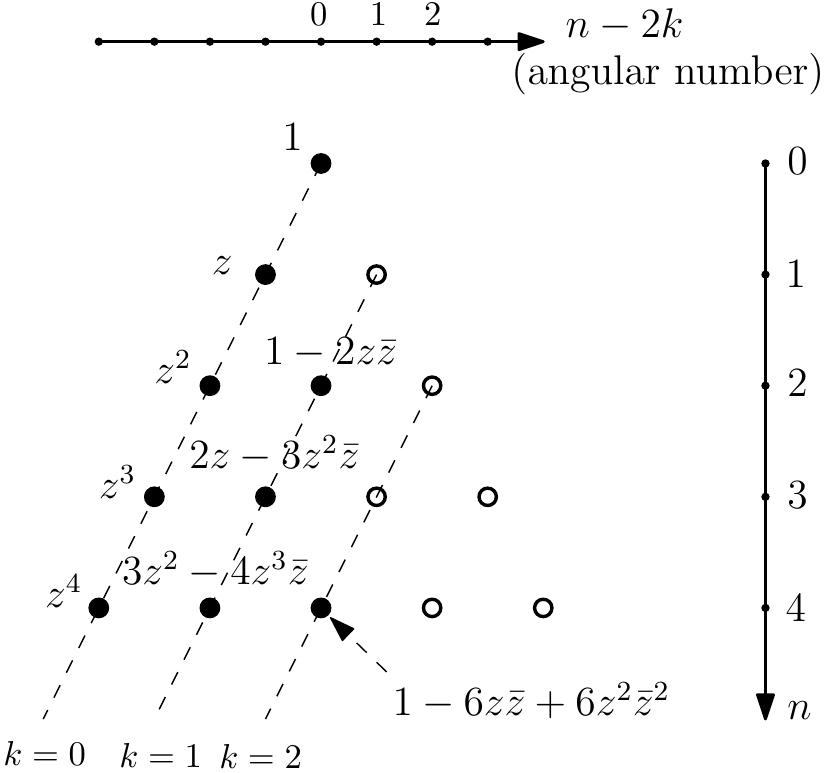}
    \caption{Structure of the Zernike polynomials in the convention of \cite{Kazantsev2004}. The ones marked '$\circ$' can be deduced from the ones marked '$\bullet$' via the formula $Z_{n,n-k} = (-1)^n \overline{Z_{n,k}}$.}
    \label{fig:Zernike}
\end{figure}

\begin{figure}[htpb]
    \centering
    \includegraphics[width=0.32\textwidth]{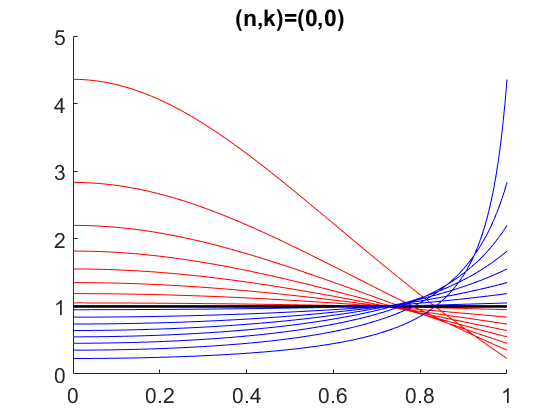}
    \includegraphics[width=0.32\textwidth]{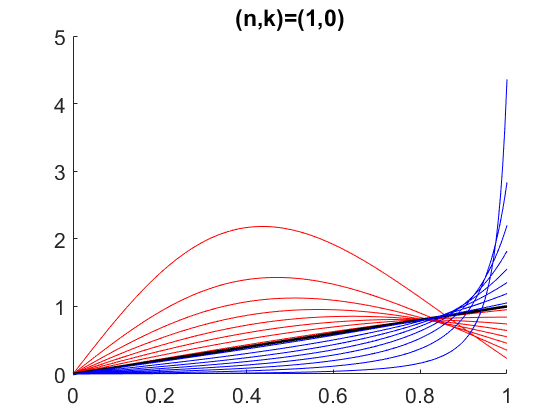}
    \includegraphics[width=0.32\textwidth]{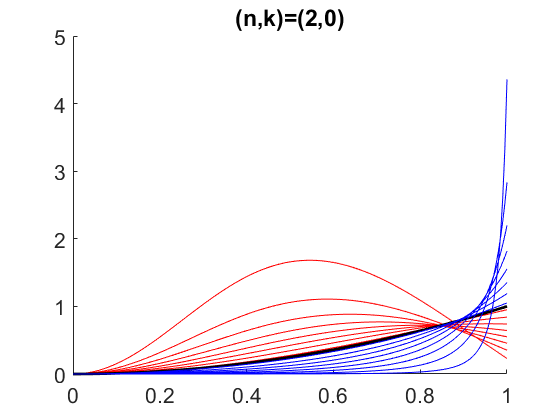} \\
    \includegraphics[width=0.32\textwidth]{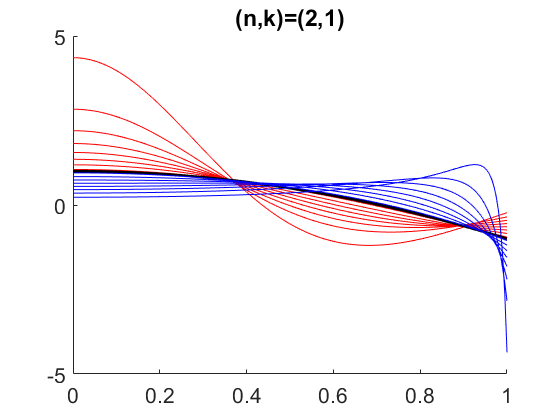}
    \includegraphics[width=0.32\textwidth]{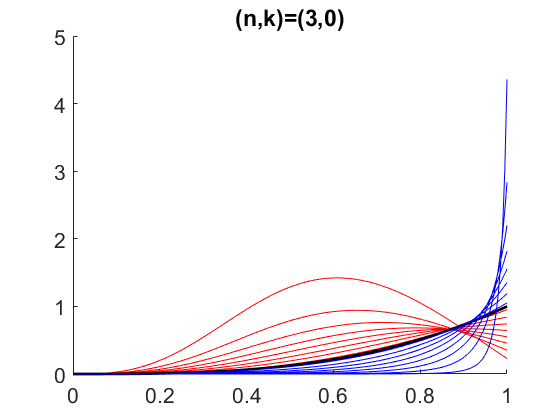}
    \includegraphics[width=0.32\textwidth]{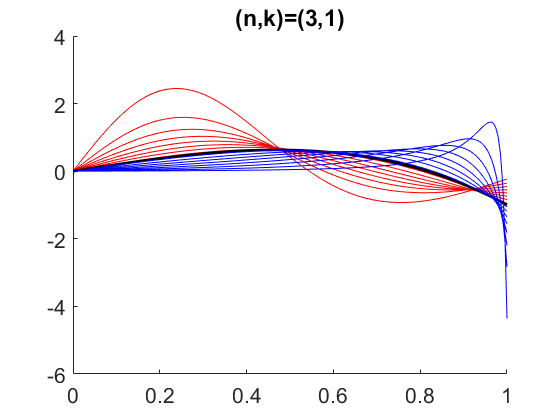} \\
    \includegraphics[width=0.32\textwidth]{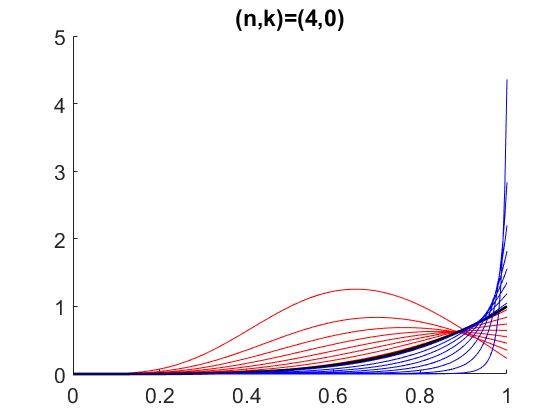}
    \includegraphics[width=0.32\textwidth]{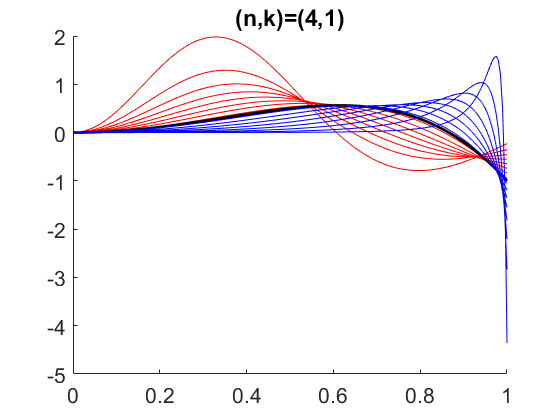}
    \includegraphics[width=0.32\textwidth]{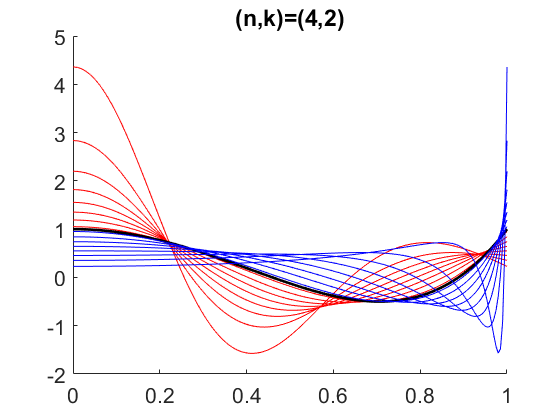}   
    \caption{Plots of the first few radial profiles of the singular functions $Z_{n,k}^\kappa (\rho e^{i\omega}) =Z_{n,k}(\rho) e^{i(n-2\kappa)\omega}$ defined in \eqref{eq:Znkk}, for various values of $\kappa\in (-1,0)$ (red) and $\kappa \in (0,1)$ (blue).}
    \label{fig:radialPlots}
\end{figure}

\paragraph{Outline.} The remainder of the article is structured as follows. In Section \ref{sec:prelims}, we first introduce the geometric models considered and compute their scattering relation, involving in particular an important function $\ss_\kappa(\alpha)$ (equal to $\alpha$ in the Euclidean case). In Section \ref{sec:3}, we construct the SVD's of the operators $P_-$ and $C_-$, which help describe the range of the geodesic X-ray transform in Theorem \ref{thm:range}. Finally, in Section \ref{sec:4}, we construct the SVD of an appropriate adjoint of $I_0$, and give a proof of Theorem \ref{thm:main}. 

\begin{remark}[On notation] In what follows, we will always work with one fixed value of $\kappa$, and all quantities are $\kappa$-dependent, whether specified in the notation or not. \F{Our choice for keeping some of the ``$\kappa$'' is mainly motivated by the fact that some equations such as \eqref{eq:Znkk} involve quantities associated with two different geometries (the one for some $\kappa\ne 0$, and the Euclidean one).} The following may give a sample of which ones generally include $\kappa$ in the notation and which ones do not: 
    \begin{align*}
	d\Sigma^2,\ g_\kappa,\ dVol_\kappa,\ \ss_\kappa,\ w_\kappa,\ \wtZ_{n,k},\ \psi_{n,k}^\kappa,\ \sigma_{n,k}^\kappa,\ C_-,\ P_-,\ SM,\ \SS,\ \SS_A,\ A_\pm,\ A_\pm^*,\ I_0,\ I_0^\sharp.
    \end{align*}
\end{remark}

\section{Preliminaries} \label{sec:prelims}

\subsection{Geometric models and their isometries}

For fixed $\kappa\in (-1,1)$, we consider the unit disk $\Dm$ equipped with the metric $g_\kappa (z) = c_\kappa(z)^{-2} |dz|^2$, $c_\kappa(z):= 1+ \kappa|z|^2$, of constant curvature $4\kappa$. Fixing $\kappa\in (-1,1)$, we will denote the unit circle bundle as 
\begin{align*}
    S\Dm = \{(z,v)\in S\Dm, \qquad |v|^2_{g_\kappa(z)} = 1\}.
\end{align*}
A point in $SM$ will be parameterized by $(z,\theta)$, where $\theta \in \Sm^1$ describes the tangent vector $v = c_\kappa(z) \binom{\cos\theta}{\sin \theta}$. The boundary $\partial S\Dm$ is parameterized in {\bf fan-beam coordinates} $(\beta,\alpha)\in \Sm^1\times \Sm^1$, where $z = e^{i\beta}$ denotes a point on $\partial M$ and $\alpha$ denotes the direction of the tangent vector $v = c_\kappa(1) e^{i(\beta+\pi+\alpha)}$ with respect to the inward normal, of direction $e^{i(\beta + \pi)}$. The boundary $\partial SM$ is equipped with a natural measure $d\Sigma^2 = c_\kappa^{-1}(1)\ d\beta\ d\alpha$, coming from restricting the Sasaki metric defined on $S\Dm$. The boundary has two distinguished components: the inward boundary $\partial_+ S\Dm = \{\alpha\in [-\pi/2, \pi/2]\}$ and the outward one $\partial_- S\Dm = \{\alpha\in [\pi/2,3\pi/2]\}$ which intersect at tangential vectors, where $\alpha= \pm \pi/2$.

For fixed $\kappa\in (0,1)$, the manifold $(\Dm,g_{\kappa})$ can be viewed as a simple surface included in the \F{Riemann sphere} $(\Cm\cup \{\infty\},g_{\kappa})$ and for $\kappa\in (-1,0)$, the manifold $(\Dm,g_{\kappa})$ can be viewed as a simple surface included in the hyperbolic space $(\Dmm_{(-\kappa)^{-1/2}}, g_{\kappa})$, where $\Dmm_{(-\kappa)^{-1/2}} = \{(x,y)\in \Rm^2,\ x^2+y^2< -\kappa^{-1}\}$. In either case, $\kappa\to 0$ recovers the standard Euclidean disk. As $|\kappa|\to 1$, simplicity breaks down for two different reasons: $(\Dm,g_{1})$ becomes a ``hemisphere'' with totally geodesic (i.e., non-convex) boundary and $(\Dm, g_{-1})$ is, up to some scalar constant\footnote{Customarily, the Poincar\'e disk carries four times this metric.}, the Poincar\'e disk, non-compact. In the latter, the interior of $\Dm$ is geodesically complete, all geodesics are asymptotically normal to the boundary and the fan-beam coordinate system breaks down.

To compute geodesics, we will use the action of isometries of either model, to move the following obvious geodesics
\begin{align}
    \begin{split}
	\kappa<0: &\qquad (z(t),\theta(t)) = \left( \frac{1}{\sqrt{-\kappa}} \tanh \left( \sqrt{-\kappa}\ t \right), 0\right), \qquad t\in \Rm, \\
	\kappa>0: &\qquad (z(t),\theta(t)) = \left( \frac{1}{\sqrt{\kappa}} \tan \left( \sqrt{\kappa}\ t \right), 0\right),  \qquad t\in \left( -\frac{\pi}{2\sqrt{\kappa}}, \frac{\pi}{2\sqrt{\kappa}} \right).	
    \end{split}
    \label{eq:geo0}    
\end{align}

One can find those isometries by conjugating the automorphisms of the Poincar\'e disk or the Riemann sphere with appropriate homotheties, which would result in subgroups of M\"obius transformations. Under this latter assumption, let us find those directly, with the immediate observation that a M\"obius transformation $T(z) = \frac{az+c}{cz+d}$ pushes forward a tangent vector $(z,\zeta)$ to $T\cdot (z,\zeta) = (T(z), T'(z) \zeta)$. We will also write $T(z) = \frac{az+b}{cz+d} = \smat{a}{b}{c}{d} (z)$ interchangeably. 

\begin{lemma}\label{lem:isometries} For $\kappa\in (0,1)$, the isometry group of $(\Cm\cup \{\infty\},g_{\kappa})$ is given by
    \begin{align}
	\Aut (\Cm\cup \{\infty\}, g_{\kappa}) = \left\{ \mat{a}{b}{-\kappa\bar b}{\bar a}, \qquad |a|^2 + \kappa |b|^2 = 1     \right\}.
	\label{eq:Aut1}
    \end{align}
    For $\kappa\in (-1,0)$, the isometry group of $(\Dmm_{(-\kappa)^{-1/2}}, g_{\kappa})$ is given by 
    \begin{align}
	\Aut (\Dmm_{(-\kappa)^{-1/2}}, g_{\kappa}) = \left\{ \mat{a}{b}{-\kappa\bar b}{\bar a}, \qquad |a|^2 +\kappa |b|^2 = 1     \right\}.
	\label{eq:Aut2}
    \end{align}   
\end{lemma}

\begin{proof} The proofs of \eqref{eq:Aut1} and \eqref{eq:Aut2} are identical. We seek a M\"obius transformation $T = \smat{a}{b}{c}{d}$ with $ad-bc = 1$ such that $g_{\kappa} (T(z)) (T'(z)\zeta, T'(z)\zeta) = g_{\kappa}(z) (\zeta,\zeta)$ for all $(z,\zeta)$. This is recast as 
    \begin{align*}
	\frac{1}{|cz+d|^2} \frac{1}{1+\kappa |T(z)|^2} = \frac{1}{ 1+ \kappa |z|^2}, 
    \end{align*}    
    which yields, for all $z$ in the space considered
    \begin{align*}
	1 + \kappa|z|^2 &= |cz+d|^2 + \kappa|az+b|^2 \\
	&= (|c|^2 + \kappa |a|^2)|z|^2 + 2 \Re \left( z(c\bar d + \kappa a\bar b) \right) + |d|^2+\kappa |b|^2.
    \end{align*}
    This is equivalent to having the relations 
    \begin{align*}
	|c|^2 + \kappa|a|^2 = \kappa, \qquad c\bar d + \kappa a\bar b = 0, \qquad |d|^2+\kappa|b|^2 = 1.
    \end{align*}
    Multiplying the second by $\bar a$ and using the first and $ad-bc = 1$, we get
    \begin{align*}
	0 = c \bar a\bar d + \kappa|a|^2 \bar b = c (1+ \bar b \bar c) + \kappa |a|^2 \bar b = c + \bar b(|c|^2 + \kappa|a|^2) = c + \kappa\bar b,
    \end{align*}
    hence $c = -\kappa\bar b$. Similarly, multiplying the same equation by $\bar c$ yields
    \begin{align*}
	0 = |c|^2 \bar b + \kappa a \bar b \bar c = |c|^2 \bar b + \kappa a \bar a\bar b - \kappa a = (|c|^2 + \kappa|a|^2) \bar d - \kappa a = \kappa \bar d -\kappa a
    \end{align*}
    So $\bar d = a$. Finally, these two relations are necessary and sufficient to describe \eqref{eq:Aut1} and \eqref{eq:Aut2}. 
\end{proof}
 
Now, given $(z_1, \theta)$ corresponding to a unit tangent vector $(z_1, c_\kappa(z_1) e^{i\theta})$, we want to find the element $T$ which maps $(0,1)$ to $(z_1, c_\kappa (z_1) e^{i\theta})$, satisfying 
\begin{align*}
    T(0) = z_1, \qquad T'(0)\cdot 1 = c_\kappa(z_1) e^{i\theta}.
\end{align*}
Seeking for an element of the form \eqref{eq:Aut1} or \eqref{eq:Aut2} immediately leads to the unique transformation 
\begin{align*}
    T(z) = T^{\kappa}_{z_1,\theta}(z)= \frac{e^{i\theta}z + z_1}{1 - \kappa e^{i\theta} \bar{z_1} z}. 
\end{align*}

\subsection{Scattering relation}

We generally define the {\bf scattering relation} $\SS:\partial S\Dm \to \partial S\Dm$ as
\begin{align}
    \SS(x,v) = \varphi_{\pm \tau(x,\pm v)} (x,v), \qquad (x,v)\in \partial_\pm SM,
    \label{eq:scatrel}
\end{align}
where $\varphi_t(x,v) = (\gamma_{x,v}(t), \dot \gamma_{x,v}(t))$ denotes the geodesic flow on a Riemannian manifold $(M,g)$ and $\tau(x,v)$ denotes the first exit time of the geodesic $\gamma_{x,v}(t)$. In our case, we now compute this relation explicitly. 

First notice by rotation-invariance and symmetry of the family of curves, that in fan-beam coordinates, one expects an expression of the form $\SS(\beta,\alpha) = (\beta + f(\alpha), \pi-\alpha)$ for some function $f$ to be determined. To determine $f$, we then set $\beta = 0$. We first compute the geodesic through the point $(1, c_\kappa(1) e^{i(\pi+\alpha)})$ with $\alpha\in (-\pi/2,\pi/2)$. From the previous section, the unique isometry mapping $(0,1)$ to that point is given by 
\begin{align*}
    T(z) = \frac{1 - e^{i\alpha} z}{1 + \kappa e^{i\alpha}z},
\end{align*}
so that $T(z(t))$ with $z(t)$ defined in \eqref{eq:geo0} is the geodesic we seek. We then solve for $|T(z(t^*))|^2 = 1$ with $t^*>0$, the point at which that geodesic exists the domain $M$, and obtain
\begin{align*}
    z(t^*) = \frac{1}{1-\kappa}\ 2\cos\alpha. 
\end{align*}
In particular, 
\begin{align*}
    T(z(t^*)) = - \frac{(1 + \kappa)\cos\alpha + i (1-\kappa)\sin\alpha}{ (1+\kappa)\cos\alpha - i(1-\kappa)\sin\alpha} = e^{i\pi} e^{2i \arg((1+\kappa)\cos\alpha + i (1-\kappa)\sin\alpha)}. 
\end{align*}
The number inside the argument belongs to the right-half plane so that we may compute that
\begin{align*}
    T(z(t^*)) = \exp \left(i\left(\pi + 2\tan^{-1} \left( \frac{1-\kappa}{1+\kappa} \tan \alpha \right)\right)\right).
\end{align*}
In particular, in fan-beam coordinates, given $(\beta,\alpha)\in \partial S\Dm$, the scattering relation is given by 
\begin{align}
    {\cal S}(\beta,\alpha) = \left( \beta + \pi + 2 \tan^{-1} \left( \frac{1-\kappa}{1+\kappa} \tan \alpha \right), \pi-\alpha\right),
    \label{eq:SA}
\end{align}
recovering the Euclidean case \cite{Monard2015a} as $\kappa\to 0$, and becoming degenerate as $\kappa\to \pm 1$.

\subsubsection{Scattering signatures}

The function $\ss = \ss_\kappa$ defined as
\begin{align}
    \ss_{\kappa} (\alpha) := \tan^{-1} \left( \frac{1-\kappa}{1+\kappa} \tan \alpha\right)
    \label{eq:ss}
\end{align}
may be thought of as a 'scattering signature' of each geometry\F{, in that it is the only function that distinguishes two circularly symmetric scattering relations on the unit disk. The function $\ss_\kappa$ describes how, fixing the endpoint $z=1$ at the boundary, the other endpoint of a geodesic moves as the inward-pointing vector above $z=1$ changes. Strikingly (though this is inconsequential for what follows), we have $\ss_\kappa \circ \ss_{-\kappa} = id$ for all $\kappa\in (-1,1)$. This can be interpreted as the fact that the geodesic 'spread' at the boundary induced by negative curvature inside the disk can be undone by precisely changing the sign of the curvature.} 

As we will work with only one fixed value of $\kappa$ at a time, we may drop the subscript $\kappa$ for conciseness. 
The scattering relation $\SS$ and antipodal scattering relation $\SS_A$ (composition of $\SS$ with the antipodal map $\alpha\mapsto \alpha+\pi$) take the form  
\begin{align}
    {\cal S}(\beta,\alpha) = \left( \beta + \pi + 2 \ss(\alpha) , \pi-\alpha\right), \qquad {\cal S}_A(\beta,\alpha) = \left( \beta + \pi + 2 \ss(\alpha) , -\alpha\right). 
    \label{eq:SA0}
\end{align}
The map $\SS_A$ is a diffeomorphism of $\partial S\Dm$, and $\partial_\pm S\Dm$ are both $\SS_A$-stable. Since integrating a function does not depend on the direction of integration, the ray transform of a function is always invariant under the pullback $\SS_A^*$. For later, we record that the function $\ss(\alpha)$ satisfies the following obvious properties: 
\begin{align*}
    \ss(\alpha+\pi) = \ss(\alpha) + \pi, \qquad \ss(-\alpha) = -\ss(\alpha), \qquad \alpha\in \Sm^1.
\end{align*}

The jacobian of $\alpha\mapsto \ss(\alpha)$ takes the expression
\begin{align*}
    \ss'(\alpha) = \frac{1}{\lambda} + \frac{\lambda^2-1}{\lambda} \frac{1}{1+\lambda^2\tan^2 \alpha}, \qquad \lambda = \frac{1-\kappa}{1+\kappa} >1.  
\end{align*}
In particular, $\frac{1}{\lambda} \le \ss'(\alpha) \le \lambda$ for all $\alpha$ and $\ss'(\alpha)$ can be used as a multiplicative weight on $L^2 (\partial_+ SM, d\Sigma^2)$ spaces, that yields an equivalent $L^2$ topology. In the Euclidean case, $\ss(\alpha) = \alpha$, and therefore no distinction is necessary. In the work that follows, it will be crucial to work with $\alpha$, $\ss(\alpha)$ or a combination of both. To this end, we now describe some important relations between the two.   

\subsubsection{Linear fractional relation between $e^{2i\alpha}$ and $e^{2i\ss(\alpha)}$ and its consequences}

An important calculation is the following: with $\ss(\alpha) = \tan^{-1} (\lambda \tan \alpha)$, $\lambda := \frac{1-\kappa}{1+\kappa}$, we compute
\begin{align}
    e^{2i\ss(\alpha)} = \frac{1+ i\lambda \tan \alpha}{1-i\lambda\tan \alpha} &= \frac{(1+\lambda)e^{i\alpha} + (1-\lambda)e^{-i\alpha}}{(1-\lambda)e^{i\alpha} + (1+\lambda)e^{-i\alpha}} \qquad \left(i\tan\alpha = \frac{e^{i\alpha}-e^{-i\alpha}}{e^{i\alpha}+e^{-i\alpha}}\right) \nonumber \\
    &= \frac{e^{2i\alpha} +\kappa}{1+\kappa e^{2i\alpha}} \qquad \left( \frac{1-\lambda}{1+\lambda} = \kappa \right),\label{eq:LFT0}
\end{align}
or in short, 
\begin{align}
    e^{2i\ss(\alpha)} = \mat{1}{\kappa}{\kappa}{1} (e^{2i\alpha}) \qquad\leftrightarrow\qquad e^{2i\alpha} = \mat{1}{-\kappa}{-\kappa}{1} (e^{2i\ss(\alpha)}).
    \label{eq:LFT1}
\end{align}

The following Lemma will be crucial. Below we will say that a function $f(\alpha)$ is a holomorphic/strictly holomorphic/antiholomorphic/strictly antiholomorphic in $e^{i\alpha}$ if its Fourier expansion in $e^{i\alpha}$ only contains non-negative/positive/non-positive/negative powers of $e^{i\alpha}$. 

\begin{lemma} \label{lem:holo}
    For any $\kappa\in (-1,1)$, the function $e^{2i\ss(\alpha)}$ is a holomorphic, even series in $e^{i\alpha}$, with average $\kappa$. As a result, for any $q> 0$, $e^{2iq\ss(\alpha)}$ is a holomorphic, even series in $e^{i\alpha}$, and for $q<0$, $e^{2iq\ss(\alpha)}$ is an anti-holomorphic, even series in $e^{i\alpha}$. 
\end{lemma}

\begin{proof} Use a geometric sum in \F{\eqref{eq:LFT0}} to obtain
    \begin{align}
	e^{2i\ss(\alpha)} = \kappa + \left( \kappa -\kappa^{-1} \right) \sum_{p=1}^\infty (-\kappa)^p e^{2ip\alpha}.
	\label{eq:geom}
    \end{align}
    The other consequences follow from the fact that products of holomorphic series are holomorphic.
\end{proof}

The relation \eqref{eq:LFT1} also turns into a relation for the cosines: 
\begin{align*}
    e^{2i\ss(\alpha)} = \frac{e^{2i\alpha}+\kappa}{1+\kappa e^{2i\alpha}} = \frac{(e^{2i\alpha}+\kappa)(1 +\kappa e^{-2i\alpha})}{|1+\kappa e^{2i\alpha}|^2} = \frac{e^{2i\alpha} +2\kappa + \kappa^2 e^{-2i\alpha}}{1+\kappa^2 + 2\kappa \cos(2\alpha)}. 
\end{align*}
Taking the real part, we obtain
\begin{align*}
    \cos(2\ss(\alpha)) = \frac{(1+\kappa^2)\cos(2\alpha) +2\kappa}{1+\kappa^2 +2\kappa \cos(2\alpha)} = \mat{1+\kappa^2}{2\kappa}{2\kappa}{1+\kappa^2} (\cos(2\alpha)),    
\end{align*}
which inverts as 
\begin{align}
    \cos(2\alpha) = \mat{1+\kappa^2}{-2\kappa}{-2\kappa}{1+\kappa^2} (\cos(2\ss(\alpha))).
    \label{eq:cos2alpha}
\end{align}
Using these relations, one may derive useful representations of the jacobian $\ss'(\alpha)$: 
\begin{align}
    \ss'(\alpha) = \frac{1-\kappa^2}{1+\kappa^2 + 2\kappa \cos(2\alpha)} = \frac{1}{1-\kappa^2} (1+\kappa^2 -2\kappa \cos(2\ss(\alpha))).
    \label{eq:sprime}
\end{align}

\subsubsection{Relation between $e^{i\alpha}$ and $e^{i\ss(\alpha)}$}

While there is no obvious relation between $e^{i\alpha}$ and $e^{i\ss(\alpha)}$ (and it is unclear whether $e^{i\ss(\alpha)}$ is holomorphic in terms of $e^{i\alpha}$), some crucial relations are to be derived. A first one is that $\sqrt{\ss'}$ can be writen as an expression of both $e^{i\alpha}$ and $e^{i\ss(\alpha)}$. 

\begin{lemma}\label{lem:sqrtjac} With $\ss(\alpha) = \ss_\kappa(\alpha)$ as given in \eqref{eq:ss}, we have
    \begin{align}
	\sqrt{\ss'(\alpha)} = \frac{1}{\sqrt{1-\kappa^2}} e^{i\alpha} (e^{-i\ss(\alpha)} -\kappa e^{i\ss(\alpha)}).
	\label{eq:sqrtjac}
    \end{align}
\end{lemma}

\begin{proof} Recall the formula
    \begin{align*}
	\ss'(\alpha) = \frac{1-\kappa^2}{(1+\kappa e^{2i\alpha})(1+\kappa e^{-2i\alpha})} = \frac{1}{1-\kappa^2} (1-\kappa e^{2i\ss(\alpha)})(1-\kappa e^{-2i\ss(\alpha)}).
    \end{align*}
    Define $f(\alpha) := e^{i\alpha} (e^{-i\ss(\alpha)} -\kappa e^{i\ss(\alpha)})$, then an immediate calculation shows that 
    \begin{align}
	f(\alpha) \overline{f(\alpha)} = (1-\kappa^2) \ss'(\alpha).
	\label{eq:ffbar}
    \end{align}
    Further, notice that 
    \begin{align*}
	\frac{f(\alpha)}{\overline{f(\alpha)}} = e^{2i\alpha} e^{-2i\ss(\alpha)} \frac{1-\kappa e^{2i\ss(\alpha)}}{1-\kappa e^{-2i\ss(\alpha)}} &= e^{2i\alpha} \mat{-\kappa}{1}{1}{-\kappa} (e^{2i\ss(\alpha)}) \\
	&= e^{2i\alpha} \mat{-\kappa}{1}{1}{-\kappa} \mat{1}{\kappa}{\kappa}{1} (e^{2i\alpha}) = \frac{e^{2i\alpha}}{e^{2i\alpha}} = 1.
    \end{align*}
    So $f$ is in fact real-valued, and using \eqref{eq:ffbar}, it is nothing but $\sqrt{1-\kappa^2} \sqrt{\ss'(\alpha)}$. 
\end{proof}

Multiplying \eqref{eq:sqrtjac} by $e^{-i\alpha}$ and identifying real and imaginary parts, we obtain relations for the sines and cosines:
\begin{align}
    \sqrt{\frac{1+\kappa}{1-\kappa}} \sqrt{\ss'(\alpha)} \cos \alpha = \cos(\ss(\alpha)), \qquad \sqrt{\frac{1-\kappa}{1+\kappa}} \sqrt{\ss'(\alpha)} \sin \alpha = \sin(\ss(\alpha)).
    \label{eq:important}
\end{align}

\section{Singular Value Decomposition of the boundary operators and moment conditions for $I_0$.} \label{sec:3}

Out of the scattering relation \eqref{eq:scatrel}, one defines operators \F{of extension from $\partial_+ SM$ to $\partial SM$ by evenness/oddness with respect to the scattering relation:} 
\begin{align*}
    A_\pm\colon L^2(\partial_+ SM, \mu\ d\Sigma^2) \to L^2(\partial SM, |\mu|\ d\Sigma^2), \quad A_\pm u(x,v) = \left\{
    \begin{array}{cc}
	u(x,v), & (x,v)\in \partial_+ SM \\
	\pm u(\SS(x,v)), & (x,v)\in \partial_- SM,
    \end{array}
    \right.
\end{align*}
with adjoints $A_\pm^* u(x,v) := u(x,v) \pm u(\SS(x,v))$ for $(x,v)\in \partial_+ SM$. For $(x,v)\in \partial SM$, the function $\mu$ is defined as $\mu(x,v) = g_x(v,\nu_x)$ with $\nu_x$ the unit inner normal to $x\in \partial M$, in particular in fan-beam coordinates, this is nothing but $\cos\alpha$. 

In the circularly symmetric case, since $\mu(\SS(x,v)) = -\mu(x,v)$, $A_\pm$ and $A_\pm^*$ are also adjoints of one another in the $L^2(\partial_+ S\Dm, d\Sigma^2) \to L^2(\partial S\Dm,d\Sigma^2)$ setting. In the smooth setting, as such extensions may generate singularities at the tangential directions, one must define, somewhat tautologically for now, 
\begin{align*}
    A_\pm &\colon C^\infty_{\alpha,\pm} (\partial_+ SM) \to C^\infty(\partial SM) \qquad \text{where}, \\
    C^\infty_{\alpha,\pm} (\partial_+ SM) &:= \{u\in C^\infty(\partial_+ SM),\ A_\pm u \in C^\infty(\partial SM)\},\end{align*}
see Appendix \ref{sec:spaces} for more detail, and for their further decompositions into spaces $C_{\alpha,\pm,\pm}^\infty(\partial_+ SM)$ in Eq. \eqref{eq:Calphaspaces}. We define the fiberwise Hilbert transform $H\colon C^\infty(\partial SM)\to C^\infty(\partial SM)$, defined in fan-beam coordinates as
\begin{align}
    Hu (\beta,\alpha) = \sum_{k\in \Zm} -i\text{sign}(k) u_k(\beta) e^{ik\alpha} \qquad \text{for}\qquad u = \sum_{k\in \Zm} u_k(\beta) e^{ik\alpha},
    \label{eq:Hilbert}
\end{align}
\F{with the convention that $\text{sign}(0) = 0$. Then} write $H = H_+ + H_-$, where $H_{+/-}$ is the restriction of $H$ onto even/odd Fourier modes. Out of these operators, we can then define two important operators
\begin{align}
    \begin{split}
	P_\pm &\colon C^\infty_{\alpha,+}(\partial_+ SM) \to C^\infty(\partial_+ SM), \qquad P_\pm := A_-^* H_\pm A_+, \\
	C_\pm &\colon C^\infty_{\alpha,-}(\partial_+ SM) \to C^\infty(\partial_+ SM), \qquad C_\pm := \frac{1}{2} A_-^* H_\pm A_-.
    \end{split}    
    \label{eq:PC}
\end{align}
One of the purposes of this section will be to compute the SVD's of $P_-$ and $C_-$ for the $L^2(\partial_+ SM, d\Sigma^2)\to L^2(\partial_+ SM, d\Sigma^2)$ topology. The relevance of these operators comes from the range characterization described in Proposition \ref{prop:PU}, which tell us that understanding the range of $I_0$ reduces to understanding the range of $P_-$ on $C_{\alpha,+,-}^\infty(\partial_+ SM)$. Moreover, understanding $C_-$ provides another range characterization for $I_0$, together with operators for projecting noisy data onto the range of $I_0$. 

In Section \ref{sec:spacescharac}, we first give a characterization of the spaces $C_{\alpha,\pm,\pm}^\infty(\partial_+ S\Dm)$ in terms of 'natural', distinguished bases. We then modify these bases in Section \ref{sec:SVD_PC} so as to construct the SVD's of $P_-$ and $C_-$. Finally in Section \ref{sec:range}, we then formulate the range characterizations of $I_0$, together with some consequences and applications.

\subsection{Description of the spaces $C_{\alpha,\pm,\pm}^\infty (\partial_+ S\Dm)$} \label{sec:spacescharac}

In cases where the scattering relation admits an explicit expression, we can construct bases for $C_{\alpha,\pm,\pm}^\infty(\partial_+ SM)$ defined in Eq. \eqref{eq:Calphaspaces} using appropriate Fourier series, ruling out some coefficients by symmetry arguments. Upon defining the family 
\begin{align}
    e_{p,\ell} (\beta,\alpha) := e^{i(p\beta + \ell \ss(\alpha))}, \qquad (\beta,\alpha)\in \partial SM, \qquad (p,\ell)\in \Zm^2,
    \label{eq:epl}
\end{align}
we can formulate the following
\begin{proposition} 
    In the models $(\Dm, g_\kappa)$, $\kappa\in (-1,1)$, the spaces $C_{\alpha,\pm}^\infty (\partial_+ S\Dm)$ are spanned\footnote{in the sense of expansions with rapid decay. \F{This decay is inherited from the rapid decay of Fourier series of smooth periodic functions, as in Eq. \eqref{eq:rapiddecay}}.} by: 
    \begin{align}
	C_{\alpha,+,+}^\infty (\partial_+ S\Dm) &= \left\langle e_{p,2q} + (-1)^p e_{p,2(p-q)}, \qquad p,q\in \Zm^2 \right\rangle, \label{eq:Cpp} \\
	C_{\alpha,+,-}^\infty (\partial_+ S\Dm) &= \left\langle e_{p,2q+1} - (-1)^p e_{p,2(p-q)-1}, \qquad p,q\in \Zm^2 \right\rangle, \label{eq:Cpm} \\
	C_{\alpha,-,+}^\infty (\partial_+ S\Dm) &= \left\langle e_{p,2q+1} + (-1)^p e_{p,2(p-q)-1}, \qquad p,q\in \Zm^2 \right\rangle, \label{eq:Cmp} \\
	C_{\alpha,-,-}^\infty (\partial_+ S\Dm) &= \left\langle e_{p,2q} - (-1)^p e_{p,2(p-q)}, \qquad p,q\in \Zm^2 \right\rangle. \label{eq:Cmm}	
    \end{align}    
\end{proposition}

\begin{proof}
    Let $u\in C^\infty(\partial SM)$. \F{Since the function $u(\beta, \ss^{-1}(\alpha))$ is smooth on the torus $\partial SM = \Sm_\beta^1 \times \Sm_\alpha^1$, it can be written as a Fourier series 
    \begin{align*}
	u(\beta, \ss^{-1}(\alpha)) = \sum_{p,\ell \in \Zm} u_{p,\ell}\ e^{i(p\beta+\ell\alpha)},
    \end{align*}
    for some coefficients $\{u_{p,\ell}\}_{p,\ell}$ with {\em rapid decay} in the sense that 
    \begin{align}
	\sup_{p,\ell\in \Zm} \left\{ |u_{p,\ell}| (1+|p|)^a (1+|\ell|)^b \right\} <\infty, \qquad \forall a,b\in \Nm.
	\label{eq:rapiddecay}
    \end{align}
}
    This implies the following expression for $u$:
    \begin{align*}
	u(\beta,\alpha) = \sum_{p,\ell\in \Zm} u_{p,\ell}\ e^{i(p\beta+\ell\ss(\alpha))}.
    \end{align*}
    Upon looking at $e_{p,\ell}$ defined in \eqref{eq:epl}, we find that 
    \begin{align}
	\SS_A^* e_{p,\ell} = (-1)^p e_{p,2p-\ell}, \qquad \SS^* e_{p,\ell} = (-1)^{p+\ell} e_{p,2p-\ell},
	\label{eq:epl_sym}
    \end{align}
    so that 
    \begin{align*}
	\SS_A^* u = \sum_{p,\ell \in \Zm} (-1)^p u_{p,2p-\ell}\ e_{p,\ell}, \qquad \SS^* u = \sum_{p,\ell\in \Zm} (-1)^{p+\ell} u_{p,2p-\ell}\ e_{p,\ell}. 
    \end{align*}
    Now fix $\sigma_1 \in \{+,-\}$ and $\sigma_2 \in \{+,-\}$. If $w\in C_{\alpha,\sigma_1,\sigma_2}^\infty(\partial_+ S\Dm)$, then $u := A_{\sigma_1} w$ satisfies 
    \begin{align*}
	u = \sigma_2 \SS_A^* u = \sigma_1 \SS^* u.	
    \end{align*}
    At the level of the Fourier coefficients, this means    
    \begin{align*}
	u_{p,\ell} \stackrel{(\star)}{=} \sigma_2 (-1)^p u_{p,2p-\ell} \stackrel{(\star\star)}{=} \sigma_1 (-1)^{p+\ell} u_{p,2p-\ell}, \qquad (p,\ell) \in \Zm^2. 
    \end{align*}
    For $\sigma_1 = \sigma_2$, equality $(\star\star)$ forces $u_{p,\ell} = 0$ for all $\ell$ odd, and using equality $(\star)$ implies \eqref{eq:Cpp} and \eqref{eq:Cmm} upon writing $\ell = 2q$. For $\sigma_1\ne \sigma_2$, equality $(\star\star)$ forces $u_{p,\ell} = 0$ for all $\ell$ even, and equality $(\star)$ implies \eqref{eq:Cpm} and \eqref{eq:Cmp} upon writing $\ell = 2q+1$. 
\end{proof}

\subsection{Singular value decompositions of $P_-$ and $C_-$} \label{sec:SVD_PC}

Recall the definitions \eqref{eq:PC} of $P_-$ and $C_-$, where according to Appendix \ref{sec:spaces}, $P_-$ is naturally defined on $C_{\alpha,+,-}^\infty(\partial_+ S\Dm)$ and $C_-$ is naturally defined on $C_{\alpha,-,+}^\infty(\partial_+ S\Dm)$.

Functions which transform well under $P_-$ or $C_-$ must be nicely compatible with both the fiberwise Hilbert transform \eqref{eq:Hilbert} and the scattering relation \eqref{eq:scatrel}. The bases displayed in \eqref{eq:Cpm} and \eqref{eq:Cmp} do the latter but not the former. These are naturally orthogonal in $L^2(\partial S\Dm, \ss'(\alpha)\ d\Sigma^2)$, and to make them orthogonal in $L^2(\partial S\Dm, d\Sigma^2)$ (a space where $iH_{-}$ is naturally self-adjoint), a natural modification is to multiply these bases by $\sqrt{\ss'(\alpha)}$. Let us then define, for $p,q \in \Zm$, 
\begin{align}
    \phi'_{p,q} := \sqrt{\ss'}\ e_{p,2q+1}, \qquad (p,q)\in \Zm^2.
    \label{eq:phipq}
\end{align}
Combining \eqref{eq:epl_sym} with the fact that
\begin{align*}
    \ss'(\alpha) = \ss'(\alpha+\pi) = \ss'(-\alpha) = \ss'(\pi-\alpha), \qquad \text{i.e., } \quad \SS_A^* (\ss') = \SS^* (\ss') = \ss',
\end{align*}
we immediately obtain for every $(p,q)\in \Zm^2$,
\begin{align*}
    \SS_A^* \phi'_{p,q} &= \SS_A^* (\sqrt{\ss'})\ \SS_A^* e_{p,2q+1} = \sqrt{\ss'} (-1)^p e_{p,2p-2q-1} = (-1)^p \phi'_{p,p-q-1}, \\
    \SS^* \phi'_{p,q} &= \SS^* (\sqrt{\ss'})\ \SS^* e_{p,2q+1} = \sqrt{\ss'} (-1)^{p+2q+1} e_{p,2p-2q-1} = - (-1)^p \phi'_{p,p-q-1}.
\end{align*}

Regarding $\phi'_{p,q}$ as fiberwise odd functions on $\partial SM$, their fiberwise Hilbert transform can be computed, using in an important way the $\sqrt{\ss'}$ factor. 

\begin{lemma}\label{lem:Hilbphi}
    For all $(p,q) \in \Zm^2$, we have $H\phi'_{p,q} = H_- \phi'_{p,q} = -i\sgn{2q+1} \phi'_{p,q}$.    
\end{lemma}

\begin{proof} For $q\ge 0$, $\phi'_{p,q} = (1-\kappa^2)^{-1/2} e^{i\alpha} e^{ip\beta} (e^{2iq\ss(\alpha)} -\kappa e^{2i(q+1)\ss(\alpha)})$ is, by virtue of Lemma \ref{lem:holo}, $e^{i\alpha}$ times a fiber-holomorphic series, so it is strictly holomorphic and as such satisfies $H\phi'_{p,q} = -i \phi'_{p,q}$. 

    For $q<0$, we write $\phi'_{p,q} = (1-\kappa^2)^{-1/2} e^{ip\beta} e^{i\alpha} (e^{-2i\ss(\alpha)} - \kappa) e^{2i(q+1)\ss(\alpha)}$.
    By virtue of Lemma \ref{lem:holo} again, the last factor is antiholomorphic, while upon complex-conjugating \eqref{eq:geom},  
    \begin{align*}
	e^{i\alpha} (e^{-2i\ss(\alpha)} -\kappa) = (\kappa-\kappa^{-1}) \sum_{p=1}^\infty (-\kappa)^p e^{i(-2p+1)\alpha}, 
    \end{align*}
    is a strictly antiholomorphic series. The product is thus strictly antiholomorphic in $e^{i\alpha}$, therefore $H\phi'_{p,q} = i \phi'_{p,q}$. The formula follows. 
\end{proof}

Constructing functions with symmetries under $\SS_A^*$, we then define 
\begin{align*}
    u'_{p,q} := (id + \SS_A^*) \phi'_{p,q} = \phi'_{p,q} + (-1)^p \phi'_{p,p-q-1}, \\
    v'_{p,q} := (id - \SS_A^*) \phi'_{p,q} = \phi'_{p,q} + (-1)^p \phi'_{p,p-q-1}.
\end{align*}
Such bases have the natural redundancies 
\begin{align*}
    u'_{p,q} = (-1)^p u'_{p,p-q-1}, \qquad v'_{p,q} = -(-1)^p v'_{p,p-q-1}. 
\end{align*}
Upon removing these redundancies in the set of indices, we can rewrite \eqref{eq:Cpm} and \eqref{eq:Cmp} as
\begin{align*}
    C_{\alpha,-,+}^\infty(\partial_+ S\Dm) &= \left\langle u'_{p,q},\ p<2q+1 \right\rangle, \\ 
    C_{\alpha,+,-}^\infty(\partial_+ S\Dm) &= \left\langle v'_{p,q},\ p\le 2q+1 \right\rangle. 
\end{align*}
Finally, we note how the basis elements $\phi'_{p,q}$ transform under $id-\SS^*$: 
\begin{align*}
    (id-\SS^*) \phi'_{p,q} = u'_{p,q} , \qquad (id-\SS^*)(-1)^p \phi'_{p,p-q-1} = u'_{p,q}.
\end{align*}

Now, given the properties satisfied by $\phi'_{p,q}$, $u'_{p,q}$, $v'_{p,q}$, the action of $H_-$ and $\SS^*$ and $\SS_A^*$ on them are formally identical as in the Euclidean case, and the same calculation as in \cite[p. 444]{Monard2015a} allows to deduce that for any $(p,q)$ in the appropriate range, 
\begin{align}
    \begin{split}
	C_- u'_{p,q} &= \frac{-i}{2} (\sgn{2q+1} + \sgn{2p-2q-1}) u'_{p,q}, \\
	P_- v'_{p,q} &= -i (\sgn{2q+1} - \sgn{2p-2q-1}) u'_{p,q}. 
    \end{split}
    \label{eq:SVD_CmPm}	    
\end{align}
Since the families $\{u'_{p,q}\}$ and $\{v'_{p,q}\}$ are orthogonal in $L^2(\partial_+ S\Dm, d\Sigma^2)$, this automatically produces the singular value decompositions of $P_-$ and $C_-$, viewed as operators from that space into itself. The statements are identical to those of the Euclidean case made in \cite[Prop. 1 and 2]{Monard2015a} (except that the definitions of $u'_{p,q}$ and $v'_{p,q}$ differ from \cite{Monard2015a} by a fixed constant). Below we denote the orthogonal splitting 
\begin{align*}
    L^2(\partial_+ S\Dm, d\Sigma^2) = \V_+ \oplus \V_-, \qquad \V_\pm := L^2(\partial_+ S\Dm, d\Sigma^2) \cap \ker (id\mp \SS_A^*).    
\end{align*}

\begin{theorem} \label{thm:SVD_PC} Given $\kappa\in (-1,1)$, let $\Dm$ be the unit disk equipped with the metric $g_\kappa$ \eqref{eq:metric} and define $P_-, C_-$ as in \eqref{eq:PC}. 
    The SVD of the operator $P_-\colon \V_- \to \V_+$ is given by: for any $(p,q)\in \Zm^2$ with $p<2q+1$,
    \begin{align*}
	P_- v'_{p,q} = \left\{
	\begin{array}{cc}
	    -2i u'_{p,q} & \text{if } q>\frac{-1}{2} \text{ and } p<q+\frac{1}{2}, \\
	    0 & \text{otherwise}.
	\end{array}
	\right.
    \end{align*}
    The eigendecomposition of $C_-\colon \V_+ \to \V_+$ is given by: for any $(p,q)\in \Zm^2$ with $p<2q$,  
    \begin{align*}
	C_- u'_{p,q} = \left\{
	\begin{array}{cc}
	    i\ u'_{p,q}, & \text{if } q<\frac{-1}{2} \text{ and } p<q+\frac{1}{2}, \\
	    -i\ u'_{p,q}, & \text{if } q>\frac{-1}{2} \text{ and } p>q+\frac{1}{2}, \\
	    0 & \text{otherwise}.
	\end{array}
	\right.
    \end{align*}
\end{theorem}

\subsection{Consequences of Theorem \ref{thm:SVD_PC}: range characterizations of $I_0$ and a projection operator} \label{sec:range}

With all the facts collected in the previous sections, we can now prove Theorem \ref{thm:range}. 

\begin{proof}[Proof of Theorem \ref{thm:range}]
    '(1) $\iff$ (2)' is Proposition \ref{prop:PU}. 

    '(2) $\implies$ (3)' comes from the fact that $C_- P_- = 0$ as readily seen from \eqref{eq:SVD_CmPm}, and '(3) $\implies$ (2)' comes from the fact that $C_-$ has zero kernel on $(\text{Ran }P_-)^\perp$ (as a subspace of $\V_+$). 

    '(3) $\iff$ (4)' is a characterization by orthogonality of $(\ker C_-)^\perp = (\text{Ran } P_-)^\perp$. The formulation in terms of functions $\psi^\kappa_{n,k}$ is obtained through the re-indexing \eqref{eq:reindex} performed in the next sections. 
\end{proof}

\paragraph{Projection of noisy data onto the range of $I_0$.}

In addition, for purposes of projection of noisy data onto the range of $I_0$, an immediate consequence of Theorem \ref{thm:SVD_PC} is the following :

\begin{theorem}\label{thm:projection}
    Let $\Dm$ be equipped with the metric $g_\kappa$ for $\kappa\in (-1,1)$ fixed, and define $C_-$ as in \eqref{eq:PC}. Then the operator $id + C_-^2$ is the $L^2(\partial_+ S\Dm, d\Sigma^2)$ orthogonal projection operator onto the range of $I_0$.   
\end{theorem}

\begin{proof} Following Theorem \ref{thm:SVD_PC}, a direct computation at the level of the eigenvectors gives:
    \begin{align*}
	(id + C_-^2) u'_{p,q} = \left\{
	    \begin{array}{cc}
		0, & \text{if } q<\frac{-1}{2} \text{ and } p<q+\frac{1}{2}, \\
		0, & \text{if } q>\frac{-1}{2} \text{ and } p>q+\frac{1}{2}, \\
		u'_{p,q} & \text{otherwise}.
	    \end{array}
	    \right.
	\end{align*}    
\end{proof}

\section{Singular Value Decomposition of the X-ray transform} \label{sec:4}

A conclusion of Theorem \ref{thm:range} is that the range of $I_0$ is spanned by 
\begin{align}
    \{u'_{p,q},\ q>-1/2,\ q>p-1/2\},
    \label{eq:RanI0}
\end{align}
an orthogonal family in $\V_+$. In what follows, the goal is to apply an appropriate adjoint for $I_0$ to the family \eqref{eq:RanI0}, and find a topology for which the functions obtained are orthogonal. Most adjoints for $I_0$ are constructed out of a distinguished one which we denote $I_0^\sharp$: it corresponds to the adjoint of $I_0 \colon L^2(\Dm, dVol_\kappa)\to L^2(\partial_+ S\Dm,\ \mu\ d\Sigma^2)$, which in our setting takes the expression
\begin{align}
    I_0^\sharp g (z) = \int_{\Sm^1} g(\beta_-(z,\theta),\alpha_-(z,\theta))\ d\theta, \qquad z\in \Dm,
    \label{eq:adjoint}
\end{align}
where $(\beta_-,\alpha_-)(z,\theta)$ are the fan-beam coordinates of the unique $g_\kappa$-geodesic passing through $(z,\theta)\in S\Dm$, or 'footpoint map'.

In what follows, we will first recall in Section \ref{sec:SVDE} what is known in the Euclidean case, before showing that combining this knowledge with our previous derivations ultimately allows to produce the SVD of the X-ray transform in Section \ref{sec:SVD2}. Proofs of some intermediary lemmas are relegated to Section \ref{sec:thetap}.

\subsection{Euclidean case - Zernike polynomials} \label{sec:SVDE}

It may be convenient to reparameterize the set \eqref{eq:RanI0} to make the Zernike basis appear, in the form that it is presented in \cite{Kazantsev2004}. Specifically, for $n\in \Nm$ and $k\in \Zm$, we reparameterize the basis of $\V_+$ as $\psi_{n,k} := \frac{(-1)^n}{4\pi} u'_{n-2k,n-k}$ instead, i.e. we have involved the change of index
\begin{align}
    (n,k)\mapsto (p,q) = (n-2k, n-k), \qquad n\in \Nm_0,\ k\in \Zm.
    \label{eq:reindex}
\end{align}
Then an immediate calculation yields
\begin{align}
    \psi_{n,k} := \frac{(-1)^n}{4\pi} e^{i(n-2k)(\beta+\alpha)} (e^{i(n+1)\alpha} + (-1)^n e^{-i(n+1)\alpha}), \qquad n\ge 0, \quad k\in \Zm,
    \label{eq:psinkeucl}
\end{align}
and we now want to compute $I_0^\sharp \left[ \frac{\psi_{n,k}}{\mu} \right]$. Together with the definition of $I_0^\sharp$ and the relations satisfied by the Euclidean footpoint map for all $(\rho e^{i\omega},\theta)\in S\Dm$:
\begin{align*}
    \beta_-(\rho e^{i\omega}, \theta) &+ \alpha_-(\rho e^{i\omega},\theta) + \pi = \theta, \\
    \beta_-\left( \rho e^{i\omega}, \theta \right) &= \beta_- (\rho,\theta-\omega)+\omega, \qquad \alpha_{-}(\rho e^{i\omega}, \theta) = \alpha_- (\rho,\theta-\omega), 
\end{align*}
we arrive at the expression 
\begin{align*}
    I_0^\sharp \left[ \frac{\psi_{n,k}}{\mu} \right] (\rho e^{i\omega}) = e^{i(n-2k)\omega} \frac{1}{2\pi} \int_{\Sm^1} e^{i(n-2k)\theta} \frac{e^{i(n+1)\alpha_-(\rho,\theta)}+(-1)^n e^{-i(n+1)\alpha_-(\rho,\theta)}}{2\cos\alpha_-(\rho,\theta)}\ d\theta.
\end{align*}
With the relation $\sin \alpha_-(\rho,\theta) = - \rho\sin\theta$, we may rewrite this as
\begin{align}
    I_0^\sharp \left[ \frac{\psi_{n,k}}{\mu} \right](\rho e^{i\omega}) = \frac{e^{i(n-2k)\omega}}{2\pi} \int_{\Sm^1} e^{i(n-2k)\theta} W_n(-\rho\sin\theta)\ d\theta, 
    \label{eq:Znk}    
\end{align}
where we have defined 
\begin{align}
    W_n(\sin\alpha) := \frac{e^{i(n+1)\alpha}+(-1)^n e^{-i(n+1)\alpha}}{2\cos\alpha}. \label{eq:Wn}
\end{align}
The functions $W_n$ are related to the Chebychev polynomials of the second kind $U_n$, specifically through the relation $W_n(t) = i^n U_n(t)$. In particular, it is immediate to check the 2-step recursion relation and initial conditions
\begin{align*}
    W_{n+1}(t) = 2it W_n(t) + W_{n-1}(t), \qquad W_0(t) = 1, \qquad W_1 (t) = 2it.
\end{align*}
By induction, the top-degree term of $W_n$ is $(2it)^n$. Fixing $n\ge 0$, we now split the calculation into two cases:

\paragraph{Case $k<0$ or $k>n$.} In light of \eqref{eq:Znk}, since $W_n$ is a polynomial of degree $n$, then $W_n \left( - \rho\sin \theta \right)$ is a trigonometric polynomial of degree $n$ in $e^{i\theta}$. In particular, if $k<0$ or $k>n$, then $|n-2k|>n$ and thus the right hand side of \eqref{eq:Znk} is identically zero. In short, we deduce 
\begin{align*}
    I_0^\sharp \left[\frac{\psi_{n,k}}{\mu}\right] = 0, \qquad n\ge 0, \qquad k<0 \text{ or } k>n.
\end{align*}

\paragraph{Case $0\le k\le n$.} For the remaining cases, we then define $Z_{n,k}:= I_0^\sharp \left[\frac{\psi_{n,k}}{\cos\alpha}\right]$, and for the sake of self-containment, we now show that the functions $\{Z_{n,k}\}_{n\ge 0,\ 0\le k\le n}$ so constructed are the Zernike basis in the convention of \cite{Kazantsev2004}, by showing that they satisfy Cauchy-Riemann systems and take the same boundary values. 

\begin{lemma} \label{lem:Znk}
    The functions $\{Z_{n,k}\}_{n\ge 0,\ 0\le k\le n}$ satisfy the following properties: For all $n\ge 0$
    \begin{align}
	\partial_{\zbar} Z_{n,0} &= 0, \qquad \partial_z Z_{n,k} + \partial_{\zbar} Z_{n,k+1} = 0\qquad (0\le k\le n-1), \qquad \partial_{z} Z_{n,n} = 0, \label{eq:Znkprop1} \\
	Z_{n,k} (e^{i\omega}) &= (-1)^k e^{i(n-2k)\omega}, \qquad 0\le k\le n,\quad \omega \in \Sm^1. \label{eq:Znkprop2}
    \end{align}        
\end{lemma}

\begin{proof} Using the relation $W_n(-t) = (-1)^n W_n(t)$, we arrive at the expression
    \begin{align}
	Z_{n,k}(\rho e^{i\omega}) &= e^{i(n-2k)\omega} \frac{(-1)^n}{2\pi}\int_{\Sm^1} e^{i(n-2k)\theta} W_n(\rho\sin\theta)\ d\theta \nonumber \\
	&= \frac{(-1)^n}{2\pi} \int_{\Sm^1} e^{i(n-2k)\theta} W_n(\rho\sin(\theta-\omega))\ d\theta. \label{eq:Zee}
    \end{align}
    
    With $\partial_z = \frac{e^{-i\omega}}{2} (\partial_\rho - \frac{i}{\rho}\partial_\omega)$ and $\partial_{\bar z} = \frac{e^{i\omega}}{2} (\partial_\rho + \frac{i}{\rho}\partial_\omega)$, we compute 
    \begin{align*}
	\partial_z (\rho\sin(\theta-\omega)) = i\frac{e^{-i\theta}}{2}, \qquad \partial_{\bar{z}} (\rho\sin(\theta-\omega)) = -i \frac{e^{i\theta}}{2}.
    \end{align*}
    Plugging these into \eqref{eq:Zee} immediately implies 
    \begin{align}
	\partial_z Z_{n,k} + \partial_{\zbar} Z_{n,k+1} &= 0, \qquad 0\le k\le n-1.
	\label{eq:CR}
    \end{align}
    In addition, we compute
    \begin{align*}
	Z_{n,0} (\rho e^{i\omega}) &= e^{in\omega} \frac{(-1)^n}{2\pi} \int_{\Sm^1} e^{in\theta} W_n(\rho\sin\theta)\ d\theta \\
	&= e^{in\omega} \frac{(-1)^n}{2\pi}  \int_{\Sm^1} e^{in\theta} (2i\rho\sin\theta)^n\ d\theta \\
	&= \rho^n e^{in\omega}  \frac{(-1)^n}{2\pi} \int_{\Sm^1} e^{in\theta} (2i\sin\theta)^n\ d\theta
    \end{align*}
    where the second equality comes from the fact that the lower-order terms of $W_n(\rho\sin\theta)$ have no harmonic content along $e^{in\theta}$. Finally, the constant is 
    \begin{align*}
	\int_{\Sm^1} e^{in\theta} (e^{i\theta} - e^{-i\theta})^n\ d\theta = \int_{\Sm^1} (e^{2i\theta}-1)^n\ d\theta = 2\pi (-1)^n.
    \end{align*}
    In short, $Z_{n,0} = \rho^n e^{in\omega} = z^n$. This also implies $\partial_{\zbar} Z_{n,0}=0$ and since we have $Z_{n,n} = (-1)^n\overline{Z_{n,0}} = (-1)^n \zbar^n$, we deduce that $\partial_z Z_{n,n} = 0$. 

    To prove the boundary condition, using that $Z_{n,k}(\rho e^{i\omega}) = e^{i(n-2k)\omega} Z_{n,k}(\rho)$, it is enough to show that $Z_{n,k}(1) = (-1)^k$ for every $n\ge 0$ and $0\le k\le n$. That this is true for $k=0$ and $k=n$ follows from the expressions just computed, and the general claim follows by induction on $n$ once the following equality is satisfied: 
    \begin{align}
	Z_{n,k}(1) = Z_{n-2,k-1}(1) - Z_{n-1,k-1}(1) + Z_{n-1,k}(1).
	\label{eq:claimZ}
    \end{align}
    To prove \eqref{eq:claimZ}, it suffices to input the recursion $W_n(\sin\theta) = 2i\sin\theta W_{n-1}(\sin\theta) + W_{n-2}(\sin\theta)$ into the expression \eqref{eq:Zee}, and to evaluate it at $\rho e^{i\omega} = 1$. 
\end{proof}

From Lemma \ref{lem:Znk}, we see that the family so defined satisfies the characterization (b) of \cite[Theorem 1]{Kazantsev2004} of the Zernike polynomials. One may see that this characterization defines the same family due the following facts: for $n\ge 0$ and $k=0$, the functions $Z_{n,k}$ in both sets agree; by induction on $k>0$, in both sets of functions, $Z_{n,k}$ satisfies a $\partial_{\zbar}$ equation with same right-hand side and same boundary condition, for which a solution is unique if it exists. 

We can then use some of the properties given in \cite{Kazantsev2004}, in particular, the following orthogonality property 
\begin{align}
    \dprod{Z_{n,k}}{Z_{n',k'}}_{L^2(\Dm)} = \frac{\pi}{n+1}\ \delta_{n,n'}\ \delta_{k,k'},
    \label{eq:Znknorm}
\end{align}
and the fact that $\left\{\frac{\sqrt{n+1}}{\sqrt{\pi}}Z_{n,k}\right\}_{n\ge 0,\ 0\le k\le n}$ is an orthonormal basis of $L^2(\Dm)$. 

\subsection{Constant curvature case - Proof of Theorem \ref{thm:main}} \label{sec:SVD2}

As in the previous section, we reparameterize the basis of $\V_+$ using $(n,k)$ indexing: for $n\in \Nm$ and $k\in \Zm$, consider $\psi^\kappa_{n,k} := \frac{(-1)^n}{4\pi} u'_{n-2k,n-k}$, which can be rewritten as
\begin{align}
    \begin{split}
	\psi^\kappa_{n,k} &= \frac{(-1)^n}{4\pi} \sqrt{\ss'(\alpha)} e^{i(n-2k) (\beta+\ss(\alpha))} g_n(\ss(\alpha)), \qquad \text{where} \\
	g_n(\ss(\alpha)) &:= (e^{i(n+1)\ss(\alpha)}+(-1)^n e^{-i(n+1)\ss(\alpha)}).	
    \end{split}
    \label{eq:psink}
\end{align}

First observe the following fact: 
\begin{lemma}\label{lem:ortho}
    The family $\{\psi^\kappa_{n,k},\ n\ge 0,\ k\in \Zm\}$ is orthogonal in $\V_+$, with norm $\|\psi^\kappa_{n,k}\|^2 =\frac{1}{4(1+\kappa)}$ for all $n\ge 0$ and $k\in \Zm$.     
\end{lemma}

\begin{proof}
    Let $(n,k)$ and $(n',k')$ given. First notice that if $n-2k\ne n'-2k'$, the inner product $(\psi^\kappa_{n,k},\psi^\kappa_{n',k'})_{d\Sigma^2}$ will vanish due to the integration of $e^{i(n-2k - (n'-2k'))\beta}$. Now assuming $n-2k=n'-2k'$, this implies that $n$ and $n'$ have the same parity. In this case, write for example $n' = n + 2\ell$ for some $\ell\ge 0$, fix $k'$ such that $n-2k = n'-2k'$, and compute
    \begin{align*}
	(\psi^\kappa_{n,k},\psi^\kappa_{n',k'})_{d\Sigma^2} &= \frac{c_\kappa(1)^{-1}}{8\pi} \int_{-\pi/2}^{\pi/2} g_{n}(\ss(\alpha)) \overline{g_{n+2\ell}}(\ss(\alpha))\ \ss'(\alpha)\ d\alpha\\
	&= \frac{c_\kappa(1)^{-1}}{8\pi} \int_{-\pi/2}^{\pi/2} g_{n}(\alpha) \overline{g_{n+2\ell}}(\alpha)\ \ d\alpha\\
	&= \frac{c_\kappa(1)^{-1}}{4\pi} \int_{-\pi/2}^{\pi/2} (\cos(2\ell \alpha) + (-1)^n \cos (2(n+\ell+1)\alpha) )\ d\alpha \\
	&= \frac{1}{4(1+\kappa)} \delta_{\ell,0},
    \end{align*}
    hence the result. 
\end{proof}

For the topology $L^2(\partial_+ S\Dm,\ d\Sigma^2)$, the adjoint of $I_0$ is given by $w\mapsto I_0^\sharp \left[\frac{w}{\mu}\right]$ with $I_0^\sharp$ defined in \eqref{eq:adjoint}. Let us then consider the functions 
\begin{align*}
    I_0^\sharp \left[ \frac{\psi^\kappa_{n,k}}{\mu} \right] (\rho e^{i\omega}) = \frac{(-1)^n}{2\pi} \int_{\Sm^1} e^{i(n-2k) (\beta_{-} + \ss(\alpha_-))} \sqrt{\ss'(\alpha_-)} \frac{e^{i(n+1)\ss(\alpha_-)} + (-1)^n e^{-i(n+1)\ss(\alpha_-)}}{2\cos(\alpha_-)}\ d\theta 
\end{align*}
where $(\beta_-,\alpha_-)$ are short for $(\beta_-(\rho e^{i\omega}, \theta),\alpha_- (\rho e^{i\omega},\theta))$, the fan-beam coordinates of the unique $g_\kappa$-geodesic passing through $(\rho e^{i\omega}, \theta)$. With the identities \eqref{eq:important}, this can be rewritten as 
\begin{align*}
    I_0^\sharp &\left[ \frac{\psi^\kappa_{n,k}}{\mu} \right] (\rho e^{i\omega}) \\
    &= \sqrt{\frac{1+\kappa}{1-\kappa}} \frac{(-1)^n}{2\pi} \int_{\Sm^1} e^{i(n-2k)(\beta_{-} + \ss(\alpha_-))} \ss'(\alpha_-) \frac{e^{i(n+1)\ss(\alpha_-)} + (-1)^n e^{-i(n+1)\ss(\alpha_-)}}{2\cos(\ss(\alpha_-))}\ d\theta \\
    &= \sqrt{\frac{1+\kappa}{1-\kappa}} \frac{(-1)^n}{2\pi} \int_{\Sm^1} e^{i(n-2k)(\beta_{-} + \ss(\alpha_-))} \ss'(\alpha_-) W_n(\sin(\ss(\alpha_-)))\ d\theta.
\end{align*}
Using the symmetries 
\begin{align*}
    \beta_- (\rho e^{i\omega}, \theta) = \beta_- (\rho,\theta-\omega) + \omega, \qquad \alpha_{-}(\rho e^{i\omega}, \theta) = \alpha_- (\rho, \theta-\omega),
\end{align*}
we obtain the expression 
\begin{align}
    I_0^\sharp &\left[ \frac{\psi^\kappa_{n,k}}{\mu} \right] (\rho e^{i\omega}) = \sqrt{\frac{1+\kappa}{1-\kappa}} \frac{(-1)^n}{2\pi} e^{i(n-2k)\omega} \int_{\Sm^1} e^{i(n-2k)(\beta_{-} + \ss(\alpha_-))} \ss'(\alpha_-) W_n(\sin(\ss(\alpha_-)))\ d\theta,
    \label{eq:tmp}
\end{align}
with $W_n$ defined in \eqref{eq:Wn}, and where $(\alpha_-,\beta_-)$ are now evaluated at $(\rho,\theta)$. We now need to make the functions $\beta_- + \ss(\alpha_-)$ and $\sin (\ss(\alpha_-))$ more explicit. Specifically, we will derive the following in the next section: 

\begin{lemma}\label{lem:btalminus}
    The following relations hold: 
    \begin{align}
	\beta_-(\rho,\theta) + \ss(\alpha_-(\rho,\theta)) + \pi &= \theta - \tan^{-1} \left(\frac{\kappa \rho^2 \sin(2\theta)}{1+\kappa\rho^2 \cos(2\theta)}\right),
	\label{eq:bmam} \\
	\frac{\sin (\ss(\alpha_-(\rho,\theta)))}{\sqrt{\ss'(\alpha_-(\rho,\theta))}} &= -\frac{\sqrt{1-\kappa^2}}{1+\kappa\rho^2}\rho \sin \theta.
	\label{eq:sinam}
    \end{align}
\end{lemma}
In light of \eqref{eq:bmam}, we want to make in \eqref{eq:tmp} the change of variable in the fiber
\begin{align}
    \theta'(\rho,\theta) := \theta - \tan^{-1} \left(\frac{\kappa\rho^2 \sin(2\theta)}{1+\kappa\rho^2 \cos(2\theta)}\right).
    \label{eq:thetap}
\end{align}
We then state two important identities, also proved in the next section: 
\begin{lemma}\label{lem:thetap}
    The change of variable $\theta\to \theta'$ in \eqref{eq:thetap} satisfies the following: 
    \begin{align}
	\frac{\partial \theta'}{\partial\theta} &= \frac{1- \kappa\rho^2}{1+\kappa\rho^2} \frac{1+\kappa}{1-\kappa} \ss'(\alpha_-(\rho,\theta)), \label{eq:jactheta} \\
	\sin\theta' &= \frac{1-\kappa\rho^2}{1+\kappa\rho^2} \sqrt{\frac{1+\kappa}{1-\kappa}} \sqrt{ \ss'(\alpha_-(\rho,\theta))} \sin\theta. \label{eq:sintheta}
    \end{align}
\end{lemma}
Combining \eqref{eq:sintheta} with \eqref{eq:sinam}, we arrive at the relation
\begin{align*}
    \sin(\ss(\alpha_-(\rho,\theta))) = - \frac{1-\kappa}{1-\kappa\rho^2} \rho\sin \theta'.
\end{align*}
Using these relations with \eqref{eq:tmp}, we then arrive at 
\begin{align}
    \frac{1-\kappa\rho^2}{1+\kappa\rho^2} \sqrt{\frac{1+\kappa}{1-\kappa}} I_0^\sharp \left[ \frac{\psi^\kappa_{n,k}}{\mu} \right] (\rho e^{i\omega})  &= \frac{e^{i(n-2k)\omega}}{2\pi} \int_{\Sm^1} e^{i(n-2k)\theta'} W_n \left( - \frac{1-\kappa}{1-\kappa\rho^2} \rho\sin \theta' \right) \frac{\partial \theta'}{\partial \theta}\ d\theta \nonumber \\
    &= \frac{e^{i(n-2k)\omega}}{2\pi} \int_{\Sm^1} e^{i(n-2k)\theta'} W_n \left( - \frac{1-\kappa}{1-\kappa\rho^2} \rho\sin \theta' \right)\ d\theta'. \label{eq:lasttmp}
\end{align}

We now split cases in a similar way as the Euclidean case. 

\paragraph{Case $k<0$ or $k>n$.} In light of \eqref{eq:lasttmp}, since $W_n$ is a polynomial of degree $n$, then the function $W_n \left( - \frac{1-\kappa}{1-\kappa\rho^2} \rho\sin \theta' \right)$ is a trigonometric polynomial of degree $n$ in $e^{i\theta'}$. In particular, if $k<0$ or $k>n$, then $|n-2k|>n$ and thus the right hand side of \eqref{eq:lasttmp} is identically zero, and we conclude that 
\begin{align}
    I_0^\sharp \left[\frac{\psi^\kappa_{n,k}}{\mu}\right] = 0, \qquad n\ge 0, \qquad k<0 \text{ or } k>n.
    \label{eq:kerI0sharp}
\end{align}

\paragraph{Case $0\le k\le n$.} When $0\le k\le n$, we then define $\wtZ_{n,k} := I_0^\sharp \left[ \frac{\psi^\kappa_{n,k}}{\mu} \right]$ and comparing \eqref{eq:lasttmp} with \eqref{eq:Znk}, we find that
\begin{align*}
    \frac{1-\kappa\rho^2}{1+\kappa\rho^2} \sqrt{\frac{1+\kappa}{1-\kappa}} \wtZ_{n,k}(\rho e^{i\omega})&= Z_{n,k} \left( \frac{1-\kappa}{1-\kappa\rho^2} \rho\ e^{i\omega} \right),
\end{align*}
in other words, for any $n\ge 0$ and $0\le k\le n$, 
\begin{align}
    \wtZ_{n,k}(\rho e^{i\omega}) = \frac{1+\kappa\rho^2}{1-\kappa\rho^2} \sqrt{\frac{1-\kappa}{1+\kappa}} Z_{n,k}  \left( \frac{1-\kappa}{1-\kappa\rho^2} \rho\ e^{i\omega} \right).
    \label{eq:wtZ}
\end{align}

\paragraph{Orthogonality of $\wtZ_{n,k}$.}
Now that we fully understand the action of $I_0^\sharp \frac{1}{\mu}$ on $\V_+$, the last question is then to find out for which topology on $\Dm$ the family $\{\wtZ_{n,k}\}$ is orthogonal. We look for a measure of the form $w(\rho)\ dVol_\kappa = w(\rho)\ \frac{\rho\ d\rho\ d\omega}{(1+\kappa\rho^2)^2}$, and want to change variable $\rho' = \frac{1-\kappa}{1-\kappa\rho^2} \rho$, with jacobian $\rho'\ d\rho' = (1-\kappa)^2 \frac{1+\kappa\rho^2}{(1-\kappa\rho^2)^3} \rho\ d\rho$, to make appear
\begin{align*}
    \int_{\Dm} \wtZ_{n,k}(\rho e^{i\omega}) &\wtZ_{n',k'}(\rho e^{i\omega}) w(\rho)\ \frac{\rho\ d\rho\ d\omega}{(1+\kappa\rho^2)^2} \\
    &= \frac{1-\kappa}{1+\kappa} \int_\Dm \frac{(1+\kappa\rho^2)^2}{(1-\kappa\rho^2)^2} Z_{n,k} (\rho' e^{i\omega}) Z_{n',k'}(\rho' e^{i\omega})\ w(\rho)\ \frac{\rho\ d\rho\ d\omega}{(1+\kappa\rho^2)^2} \\
    &= \frac{1}{1-\kappa^2} \int_\Dm Z_{n,k} (\rho' e^{i\omega}) Z_{n',k'}(\rho' e^{i\omega})\ w(\rho)\ \frac{(1-\kappa)^2 \rho\ d\rho\ d\omega}{(1-\kappa\rho^2)^2}. 
\end{align*}
In light of the jacobian, the change $\rho\to \rho'$ will land in the Euclidean volume form if $w(\rho) = \frac{1+\kappa\rho^2}{1-\kappa\rho^2}$. Assuming this is the case, we obtain, upon using \eqref{eq:Znknorm},
\begin{align*}
    \int_{\Dm} \wtZ_{n,k}(\rho e^{i\omega}) \wtZ_{n',k'}(\rho e^{i\omega}) w(\rho)\ \frac{\rho\ d\rho\ d\omega}{(1+\kappa\rho^2)^2} &= \frac{1}{1-\kappa^2} \int_\Dm Z_{n,k} (\rho' e^{i\omega}) Z_{n',k'}(\rho' e^{i\omega})\ \rho'\ d\rho'\ d\omega \\
    &= \frac{1}{1-\kappa^2} \frac{\pi}{n+1}  \delta_{n,n'}\ \delta_{k,k'}.
\end{align*} 

Now Theorem \ref{thm:SVDI0sharp} below and the proof of Theorem \ref{thm:main} will be based on the following observation: let $(\H_1,\|\cdot\|_1)$, $(\H_2,\|\cdot\|_2)$ be two Hilbert spaces and $A\colon \H_1\to \H_2$ be a bounded operator; if there exist two complete orthogonal systems $\{x_n\}$ in $\H_1$ and $\{y_n\}$ in $\H_2$ such that $A x_n = y_n$ for all $n$, then the singular value decomposition of $A$ is $(x_n/\|x_n\|_{1}, y_n/\|y_n\|_2, \|y_n\|_2/\|x_n\|_1)_n$. This also implies that the SVD of the adjoint $A^*$ is $(y_n/\|y_n\|_2, x_n/\|x_n\|_{1}, \|y_n\|_2/\|x_n\|_1)_n$.

Based on this observation and the earlier calculations, we can formulate the following result: 

\begin{theorem} \label{thm:SVDI0sharp} Let $\kappa\in (-1,1)$. Define the weight $w_\kappa(z):= \frac{1+\kappa|z|^2}{1-\kappa|z|^2}$ for $z\in \Dm$. Then the operator 
    \begin{align*}
	I_0^\sharp \frac{1}{\mu} \colon \V_+ \to L^2(\Dm, w_\kappa\ dVol_\kappa)	
    \end{align*}
    has kernel
    \begin{align*}
	\ker I_0^\sharp \frac{1}{\mu} = \text{span} \{ \psi_{n,k}^\kappa, \quad n\ge 0,\ k\in \Zm\backslash \{0,1,\dots,n\}\}
    \end{align*}
    and its restriction to the orthocomplement of that kernel has SVD $(\widehat{\psi^\kappa_{n,k}}, \widehat{\wtZ_{n,k}}, \sigma^\kappa_{n,k})_{n\ge 0,\ 0\le k\le n}$, where
    \begin{align*}
	\widehat{\psi^\kappa_{n,k}} = \frac{\psi^\kappa_{n,k}}{\|\psi^\kappa_{n,k}\|} = 2\sqrt{1+\kappa}\ \psi^\kappa_{n,k}, \qquad \widehat{\wtZ_{n,k}} = \frac{\wtZ_{n,k}}{\|\wtZ_{n,k}\|} = \frac{\sqrt{n+1}}{\sqrt{\pi}}\sqrt{1-\kappa^2}\ \wtZ_{n,k}, 
    \end{align*}
    and where the spectral values equal
    \begin{align*}
	\sigma^\kappa_{n,k} = \frac{\|\wtZ_{n,k}\|}{\|\psi^\kappa_{n,k}\|} = \frac{1}{\sqrt{1-\kappa}}\frac{2\sqrt{\pi}}{\sqrt{n+1}}, \qquad n\ge 0, \qquad 0\le k\le n.
    \end{align*}
\end{theorem}

The proof of Theorem \ref{thm:main} now becomes straightforward.

\begin{proof}[Proof of Theorem \ref{thm:main}] In light of Theorem \ref{thm:SVDI0sharp}, the SVD of the adjoint of $I_0^\sharp \frac{1}{\mu}$ just consists of interchanging the families $\widehat{\psi^\kappa_{n,k}}$, $\widehat{\wtZ_{n,k}}$, and this is the operator we are interested in. We now compute
    \begin{align*}
	\left(f, I_0^\sharp \left[\frac{g}{\mu}\right]\right)_{w_\kappa\ dVol_\kappa} = \left(w_\kappa f, I_0^\sharp \left[\frac{g}{\mu}\right]\right)_{dVol_\kappa} = \left( I_0(w_\kappa f), \frac{g}{\mu} \right)_{\mu d\Sigma^2} = \left( I_0(w_\kappa f), g \right)_{d\Sigma^2}. 
    \end{align*}
    In other words, the adjoint of the operator $I_0^\sharp \frac{1}{\mu}\colon L^2(\partial_+ S\Dm, d\Sigma^2)\to L^2(\Dm,w_\kappa\ dVol_\kappa)$ is the operator 
    \begin{align*}
	A\colon L^2(\Dm,w_\kappa\ dVol_\kappa)\to L^2(\partial_+ S\Dm, d\Sigma^2), \qquad Af :=  I_0(w_\kappa f).
    \end{align*}
    In particular, the relation $A\ \widehat{\wtZ_{n,k}} = \sigma^\kappa_{n,k}\ \widehat{\psi^\kappa_{n,k}}$ implies $I_0 \left( w_\kappa \widehat{\wtZ_{n,k}} \right) = \sigma^\kappa_{n,k}\ \widehat{\psi^\kappa_{n,k}}$ for all $n,k$. Now, given $f\in w_\kappa L^2(\Dm, w_\kappa\ dVol_\kappa)$, $\frac{f}{w_\kappa}$ expands into the basis $\widehat{\wtZ_{n,k}}$,
    \begin{align*}
	\frac{f}{w_\kappa} = \sum_{n\ge 0}\sum_{k=0}^n a_{n,k} \widehat{\wtZ_{n,k}}, \quad\text{where}\quad  a_{n,k} = \left(\frac{f}{w_\kappa}, \widehat{\wtZ_{n,k}}\right)_{w_\kappa\ dVol_\kappa} = \left( f, \widehat{\wtZ_{n,k}} \right)_{dVol_\kappa}.
    \end{align*}
    Then we compute directly 
    \begin{align*}
	I_0 f = I_0 \left( \sum_{n,k} a_{n,k} w_\kappa \widehat{\wtZ_{n,k}} \right) = \sum_{n,k} a_{n,k} A \wtZ_{n,k} = \sum_{n,k} a_{n,k} \sigma^\kappa_{n,k} \widehat{\psi^\kappa_{n,k}}.  
    \end{align*}
    hence the result.
\end{proof}

\subsection{Proof of Lemmas \ref{lem:btalminus} and \ref{lem:thetap}} \label{sec:thetap}

\begin{proof}[Proof of Lemma \ref{lem:btalminus}]
    We will compute $e^{i(\beta_- + \ss(\alpha_-))}$ and $\sin(\ss(\alpha_-))$. The first quantity (or rather, its square) admits a rather simple expression. The way to arrive there is as follows: the unique $g_\kappa$-geodesic passing through $(\rho, c_\kappa(\rho) e^{i\theta})$ has (non unit speed) equation
    \begin{align*}
	T(x) = \frac{e^{i\theta}x + \rho}{1- \kappa e^{i\theta} \rho x}, 
    \end{align*}
    for $x\in \Rm$ if $\kappa\in [0,1)$ and $|x|\le (-\kappa)^{-1/2}$ if $\kappa\in (-1,0)$. The endpoints in the unit disk are for $|T(x)|^2 = 1$, which yields the quadratic equation
    \begin{align*}
	0 = x^2 + 2x\rho\cos\theta \frac{1+\kappa}{1-\kappa^2 \rho^2} + \frac{\rho^2-1}{1-\kappa^2 \rho^2} =: x^2 - Sx + P.
    \end{align*}
    By definition of the scattering relation, the two roots $x_{\pm}$ are such that $T(x_-) = e^{i\beta_-}$ and $T(x_+) = e^{i(\beta_- + 2\ss(\alpha_-) + \pi)}$, in particular, we obtain that 
    \begin{align*}
	-e^{2i(\beta_- + \ss(\alpha_-))} = T(x_+)T(x_-) &= \frac{e^{i\theta}x_+ + \rho}{1-\kappa e^{i\theta} \rho x_+} \frac{e^{i\theta}x_- + \rho}{1-\kappa e^{i\theta}\rho x_-} \\
	&= \frac{e^{2i\theta}P + \rho e^{i\theta} S + \rho^2}{1 - \kappa e^{i\theta} \rho S + \kappa^2 e^{2i\theta} \rho^2 P} \\
	&= - e^{2i\theta} \frac{1+\kappa\rho^2 e^{-2i\theta}}{1+\kappa\rho^2 e^{2i\theta}}.
    \end{align*}
    This yields the relation
    \begin{align*}
	2(\beta_- + \ss(\alpha_-)) = 2\left(\theta - \tan^{-1} \frac{\kappa \rho^2 \sin(2\theta)}{1+\kappa\rho^2 \cos(2\theta)}\right),
    \end{align*}
    which determines $\beta_- + \ss(\alpha_-)$ up to an additive $\pi$ term. With the Euclidean relation $\beta_- + \alpha_- + \pi = \theta$, we deduce the relation \eqref{eq:bmam}. 
    
    We now derive a formula for $\sin(\ss(\alpha_-))$. Since the surrounding space has constant curvature $4\kappa$, it is convenient to define the {\it weighted sine function} $\sin_{4\kappa}$ as follows:  
    \begin{align*}
	\sin_{4\kappa}(x)  = x - \frac{(4\kappa) x^3}{3!} + \frac{(4\kappa)^2x^5}{5!} - \frac{(4\kappa)^3x^7}{7!} + \cdots = \left\{
	\begin{array}{cc}
	    \frac{1}{2\sqrt{\kappa}} \sin(2\sqrt{\kappa}\ x), & \kappa>0, \\
	    \frac{1}{2\sqrt{-\kappa}} \sinh(2\sqrt{-\kappa}\ x), & \kappa<0.
	\end{array}
	\right.	
    \end{align*}
    Such a function appears in the \textbf{law of sines} for a $g_\kappa$-geodesic triangle of geodesic sidelengths $(a,b,c)$ and opposite angles $(A,B,C)$, namely we have 
    \begin{align}
	{\frac {\sin A}{\sin_{4\kappa} a}}={\frac {\sin B}{\sin_{4\kappa}b}}={\frac {\sin C}{\sin_{4\kappa}c}},
	\label{eq:sinerule}
    \end{align}
    see \cite{Katok1992}. Denoting by $d_\kappa(z_1,z_2)$ the $g_\kappa$-geodesic distance between $z_1$ and $z_2$, it follows directly from \eqref{eq:geo0} that for $\rho\in [-1,1]$
    \begin{align*}
	d_\kappa (\rho,0) = \left\{
	\begin{array}{cc}
	    \frac{1}{\sqrt{-\kappa}} \tanh^{-1} \left( \sqrt{-\kappa}\ \rho \right), & \kappa\in (-1,0), \\
	    \frac{1}{\sqrt{\kappa}} \tan^{-1} \left( \sqrt{\kappa}\ \rho \right), & \kappa\in (0,1),
	\end{array}
	\right.
    \end{align*} 
    and by rotation invariance, $d_\kappa(z_1,0) = d_\kappa(|z_1|,0)$. In particular, trigonometric identities imply in all cases that 
    \begin{align*}
	\sin_{4\kappa} (d_\kappa (\rho,0)) = \frac{\rho}{1+\kappa\rho^2}. 
    \end{align*}    
    Applying the sine rule \eqref{eq:sinerule} to the geodesic triangle with vertices $(0,\rho,e^{i\beta_-(\rho,\theta)})$, we obtain
    \begin{align*}
	\frac{\sin(-\alpha_-(\rho,\theta))}{\sin_{4\kappa}(d_\kappa(\rho,0))} = \frac{\sin\theta}{\sin_{4\kappa}(d_\kappa (e^{i\beta_-(\rho,\theta)}, 0))} = \frac{\sin\theta}{\sin_{4\kappa}(d_\kappa (1, 0))},
    \end{align*}
    and we obtain
    \begin{align*}
	\sin (-\alpha_{-}) = \frac{ \sin_{4\kappa}(d_\kappa(\rho,0))}{ \sin_{4\kappa}(d_\kappa(1,0))}\sin \theta = \frac{1+\kappa}{1+\kappa\rho^2} \rho\sin \theta,
    \end{align*}
    and hence $\sin (\alpha_-(\rho,\theta)) = - \frac{1+\kappa}{1+\kappa\rho^2} \rho \sin \theta$. Combined with \eqref{eq:important}, we arrive at \eqref{eq:sinam}.    
\end{proof}

\begin{proof}[Proof of Lemma \ref{lem:thetap}]
    We first connect the expression $\ss'(\alpha_-(\rho,\theta))$ with $\sin \theta$: 
    \begin{align*}
	\ss'(\alpha_-) &= \frac{1}{1-\kappa^2} (1+\kappa^2 - 2\kappa \cos(2\ss(\alpha_-))) \\
	&= \frac{1}{1-\kappa^2} \left( (1-\kappa)^2 + 4\kappa \sin^2(\ss(\alpha_-)) \right) \\
	&\stackrel{\eqref{eq:sinam}}{=} \frac{1}{1-\kappa^2} \left( (1-\kappa)^2 + 4\kappa \ss'(\alpha_-) \frac{1-\kappa^2}{(1+\kappa\rho^2)^2} \rho^2 \sin^2\theta\right)  \\
	&= \frac{1-\kappa}{1+\kappa} + \ss'(\alpha_-) \frac{4\kappa\rho^2}{(1+\kappa\rho^2)^2} \sin^2\theta.
    \end{align*}
    Solving for $\ss'(\alpha_-)$ we arrive at
    \begin{align}
	\ss'(\alpha_-(\rho,\theta)) = \frac{1-\kappa}{1+\kappa} \frac{(1+\kappa\rho^2)^2}{(1+\kappa\rho^2)^2 - 4\kappa\rho^2 \sin^2\theta}.
	\label{eq:tmpsintheta}
    \end{align}
    To obtain \eqref{eq:jactheta}, differentiate the relation $e^{2i\theta'} = \frac{e^{2i\theta}+\kappa\rho^2}{1+\kappa\rho^2 e^{2i\theta}}$ to obtain
    \begin{align*}
	e^{2i\theta'} \frac{\partial \theta'}{\partial \theta} = \frac{1-\kappa^2\rho^4}{(1+\kappa\rho^2 e^{2i\theta})^2} e^{2i\theta}. 
    \end{align*}
    Then
    \begin{align*}
	\frac{\partial \theta'}{\partial \theta} = \frac{1-\kappa^2\rho^4}{(1+\kappa\rho^2 e^{2i\theta})^2} e^{2i\theta} \frac{1+\kappa\rho^2 e^{2i\theta}}{e^{2i\theta}+\kappa\rho^2} = \frac{1-\kappa^2\rho^4}{(1+\kappa\rho^2 e^{2i\theta})(1+\kappa\rho^2 e^{-2i\theta})} = \frac{1-\kappa^2\rho^4}{(1+\kappa\rho^2)^2 - 4\kappa\rho^2 \sin^2\theta},
    \end{align*}
    and \eqref{eq:jactheta} follows from using \eqref{eq:tmpsintheta}. 

    Now to relate $\sin\theta$ and $\sin\theta'$, from the relation 
    \begin{align*}
	e^{2i\theta'} = \frac{e^{2i\theta}+\kappa\rho^2}{1+\kappa\rho^2 e^{2i\theta}} = \frac{(1+\kappa^2\rho^4)\cos(2\theta) + 2\kappa\rho^2 + i(1-\kappa^2\rho^4) \sin(2\theta)}{1+\kappa^2\rho^4 +2\kappa\rho^2\cos(2\theta)},
    \end{align*}
    whose real part gives
    \begin{align*}
	\cos(2\theta') = \mat{1+\kappa^2\rho^4}{2\kappa\rho^2}{2\kappa\rho^2}{1+\kappa^2\rho^4} (\cos(2\theta)). 
    \end{align*}
    Together with the relation $\cos(2\theta) = \mat{-2}{1}{0}{1}(\sin^2\theta)$, this implies the relation
    \begin{align*}
	\sin^2\theta' &= \mat{-1}{1}{0}{2} \mat{1+\kappa^2\rho^4}{2\kappa\rho^2}{2\kappa\rho^2}{1+\kappa^2\rho^4}\mat{-2}{1}{0}{1} (\sin^2\theta) \\
	&= \frac{(1-\kappa\rho^2)^2}{(1+\kappa\rho^2)^2 - 4\kappa\rho^2\sin^2\theta}\sin^2\theta \\
	&\stackrel{\eqref{eq:tmpsintheta}}{=} \frac{1+\kappa}{1-\kappa} \frac{(1-\kappa\rho^2)^2}{(1+\kappa\rho^2)^2} \ss'(\alpha_-) \sin^2\theta.
    \end{align*}
    Together with the fact that $\sin\theta$ and $\sin\theta'$ have simultaneously the same sign, \eqref{eq:sintheta} follows upon taking squareroots.    
\end{proof}

\appendix
\section{Spaces $C^\infty_{\alpha,\pm,\pm}(\partial_+ SM)$, operators $P_\pm$, $C_\pm$ and a refinement of the Pestov-Uhlmann range characterization} \label{sec:spaces}

In this section, we work on a general simple surface $(M,g)$ with inward boundary $\partial_+ SM$. The objects of study are the geodesic X-ray transforms $I_0 \colon C^\infty(M) \to C^\infty(\partial_+ SM)$ and $I_1 \colon C^\infty(M; TM) \to C^\infty(\partial_+ SM)$, defined for any $(x,v)\in \partial_+ SM$ as 
\begin{align*}
    I_0 f (x,v) := \int_0^{\tau(x,v)} f(\gamma_{x,v}(t))\ dt, \qquad I_1 h (x,v) := \int_0^{\tau(x,v)} \langle h(\gamma_{x,v}(t)), \dot\gamma_{x,v}(t) \rangle_g \ dt,
\end{align*}
where $f$ is a smooth function, $h$ is a smooth vector field, $(\gamma_{x,v}(t),\dot\gamma_{x,v}(t))$ is the unit speed geodesic with $(\gamma_{x,v}(0), \dot \gamma_{x,v}(0))= (x,v)$, and $\tau(x,v)$ is its first exit time. 

The Pestov-Uhlmann range characterization of $I_0$ and $I_1$ appearing in \cite[Theorem 4.4]{Pestov2004} relates the ranges of $I_0$ and $I_1$ with those of $P_-$ and $P_+$ as defined on 
\begin{align}
    C_\alpha^\infty(\partial_+ SM) := \{ u \in C^\infty(\partial_+ SM), \qquad A_+ u \in C^\infty(SM) \}.
    \label{eq:Calphainf}
\end{align}

We would like to restrict $C_\alpha^\infty(\partial_+ SM)$ to a 'half'-subspace incorporating a natural symmetry associated to whether one is integrating a function or a one-form. Namely, a function $u$ in the range of $I_0$ satisfies $\SS_A^* u = u$ and a function $u$ in the range of $I_1$ satisfies $\SS_A^* u = -u$. One must also encode whether extension from $\partial_+ SM$ to $\partial SM$ through $A_{\pm}$ produces smooth functions. 

To this effect, we then define
\begin{align*}
    C_{\alpha,\pm}^\infty (\partial_+ SM) &:= \{u\in C^\infty (\partial_+ SM),\quad A_\pm u \in C^\infty(\partial SM)\}.
\end{align*}
Thus, $C^\infty_{\alpha,+}(\partial_+ SM)$ coincides with $C_\alpha^\infty(\partial_+ SM)$ as defined in \cite{Pestov2004}. 

\begin{lemma}\label{lem:pullback} 
    The spaces $C_{\alpha,\pm}^\infty(\partial_+ SM)$ are stable under the pull-back $\SS_A^*$.     
\end{lemma}

\begin{proof} The map $\SS_A$ is the composition of the scattering relation $\SS$ and the antipodal map $(x,v)\mapsto (x,-v)$, as such it can be regarded as a smooth diffeomorphism of $\partial SM$, thus $\SS_A^*$ can be viewed as an operator on $C^\infty(\partial_+ SM)$ or on $C^\infty(\partial SM)$. Moreover, we have the relations $\SS_A^* A_\pm = A_\pm \SS_A^*$. In particular, if $w\in C_{\alpha,\pm}^\infty(\partial_+ SM)$, then $A_\pm w$ is smooth on $\partial SM$. Then so is $\SS_A^* A_\pm w = A_\pm (\SS_A^* w)$, which exactly means that $\SS_A^* w\in C_{\alpha,\pm}^\infty(\partial_+ SM)$.     
\end{proof}
Lemma \ref{lem:pullback} justifies that we can now write the direct sum decompositions: 
\begin{align*}
    C_{\alpha,\pm}^\infty(\partial_+ SM) = C_{\alpha,\pm,+}^\infty(\partial_+ SM) \oplus C_{\alpha,\pm,-}^\infty(\partial_+ SM),
\end{align*}
where we have defined 
\begin{align}
    \begin{split}
	C_{\alpha,+,\pm}^\infty (\partial_+ SM) &:= \{u\in C_{\alpha,+}^\infty (\partial_+ SM),\quad \SS_A^* u = \pm u\}, \\    
	C_{\alpha,-,\pm}^\infty (\partial_+ SM) &:= \{u\in C_{\alpha,-}^\infty (\partial_+ SM),\quad \SS_A^* u = \pm u\}.	
    \end{split}
    \label{eq:Calphaspaces}    
\end{align}
Each decomposition is produced through the equality $w = w_+ + w_- = \frac{1}{2} (id + \SS_A^*) w + \frac{1}{2} (id - \SS_A^*)w$ which, thanks to Lemma \ref{lem:pullback}, produces summands in the correct spaces. Note that we can also characterize these spaces as
\begin{align*}
    C_{\alpha,+,\pm}^\infty (\partial_+ SM) &= \{u\in C_{\alpha,+}^\infty (\partial_+ SM),\ A_+ u \text{ is fiberwise even/odd}\}, \\
    C_{\alpha,-,\pm}^\infty (\partial_+ SM) &= \{u\in C_{\alpha,-}^\infty (\partial_+ SM),\ A_- u \text{ is fiberwise odd/even}\}.
\end{align*}
Recall then the definitions of the boundary operators
\begin{align*}
    P_\pm = A_-^* H_\pm A_+, \qquad C_{\pm} = \frac{1}{2} A_-^* H_\pm A_-.
\end{align*}
The spaces above provide natural smooth functional settings for these operators: 
\begin{itemize}
    \item the operators $P_\pm$ are naturally defined on $C_{\alpha,+}^\infty(\partial_+ SM)$ and in the direct decomposition $w = w_+ + w_-$, (where $\SS_A^* w_\pm = \pm w_\pm$), we get: 
	\begin{align*}
	    P_+ w &= P_+ w_+ \in \ker(id + \SS_A^*) \qquad (P_+ w_- = 0), \\
	    P_- w &= P_- w_- \in \ker(id - \SS_A^*) \qquad (P_- w_+ = 0).
	\end{align*}
    \item the operators $C_\pm$ are naturally defined on $C_{\alpha,-}^\infty (\partial_+ SM)$ and in the direct decomposition $w = w_+ + w_-$, (where $\SS_A^* w_\pm = \pm w_\pm$), we get: 
	\begin{align*}
	    C_{+} w &= C_+ w_- \in \ker (id + \SS_A^*) \qquad (C_+ w_+ = 0), \\
	    C_- w &= C_- w_+ \in \ker (id -\SS_A^*) \qquad (C_- w_- = 0).
	\end{align*}
\end{itemize}

The observations about the action of $P_\pm$ allows us to refine the Pestov-Uhlmann range characterization \cite[Theorem 4.4]{Pestov2004} as follows: 

\begin{proposition}\label{prop:PU} Let $(M,g)$ be a simple Riemannian surface with boundary. Then 

    (i) A function $u\in C^\infty(\partial_+ SM)$ belongs to the range of $I_0$ if and only if $u = P_-w$ for some $w\in C_{\alpha,+,-}^\infty(\partial_+ SM)$. 

    (ii) A function $u\in C^\infty(\partial_+ SM)$ belongs to the range of $I_1$ if and only if $u = P_+ w$ for some $w\in C_{\alpha,+,+}^\infty(\partial_+ SM)$.     
\end{proposition}

\begin{proof} We prove (i) as (ii) is similar. The usual characterization produces $v\in C_{\alpha,+}^\infty (\partial_+ SM)$ such that $u = P_- v$. Writing $v = v_+ + v_-$, we have that $u = P_{-} (v_+ + v_-) = P_- v_-$ where $v_- \in C_{\alpha,+,-}^\infty(\partial_+ SM)$. Thus $w := v_-$ fulfills (i).
\end{proof}

\bibliographystyle{siam}
\bibliography{../../../000-TexTouch/bibliography/bibliography}

\end{document}